\documentclass[preprint,11pt]{elsarticle}

\usepackage{amsmath,amsfonts,amssymb,graphicx,listings,array,amsthm}
\usepackage{wrapfig, booktabs, xcolor, subcaption, chngcntr}
\usepackage{multirow, geometry, bm, commath, hyperref}
\usepackage{algorithm,algpseudocode,float}

\makeatletter
\def\ps@pprintTitle{%
  \let\@oddhead\@empty
  \let\@evenhead\@empty
  \def\@oddfoot{\reset@font\hfil\thepage\hfil}
  \let\@evenfoot\@oddfoot
}
\makeatother

\makeatletter
\newenvironment{breakablealgorithm}
  {
   \begin{center}
     \refstepcounter{algorithm}
     \hrule height.8pt depth0pt \kern2pt
     \renewcommand{\caption}[2][\relax]{
       {\raggedright\textbf{\ALG@name~\thealgorithm} ##2\par}%
       \ifx\relax##1\relax 
         \addcontentsline{loa}{algorithm}{\protect\numberline{\thealgorithm}##2}%
       \else 
         \addcontentsline{loa}{algorithm}{\protect\numberline{\thealgorithm}##1}%
       \fi
       \kern2pt\hrule\kern2pt
     }
  }{
     \kern2pt\hrule\relax
   \end{center}
  }
\makeatother

\usepackage{tikz,tikzscale,pgfplots}
\pgfplotsset{compat=newest}

\DeclareCaptionFormat{myformat}{\fontsize{9}{10}\selectfont#1#2#3}
\numberwithin{equation}{section}
\newcounter{rowno}
\newcommand{\mi}{\texttt{m}}
\renewcommand{\ni}{\texttt{n}}
\newcommand{\mn}{\mi\ni}
\newcommand{\mnp}{\mi\ni,\v p}
\renewcommand{\v}[1]{\textbf{#1}}
\newcommand{\p}{\partial}
\newcommand{\wt}{\widetilde}
\newcommand{\wh}{\widehat}
\newcommand{\bx}{\textbf{x}}

\newcommand{\ii}{\textrm{i}}
\newcommand{\erfc}{\text{erfc}}
\newcommand{\erf}{\text{erf}}

\newcommand{\order}[1]{{\cal O}(#1)}

\newcommand{\n}{\textsf{n}}

\newcommand{\e}{{\mathcal{E}}}

\renewcommand{\L}{\mathrm{L}}
\renewcommand{\S}{\mathrm{S}}
\renewcommand{\c}{\mathrm{c}}
\renewcommand{\P}{\mathrm{P}}
\newcommand{\E}{\mathrm{E}}
\newcommand{\F}{\mathrm{\v F}}
\newcommand{\SE}{{_\mathrm{SE}}}
\newcommand{\SPME}{{_\mathrm{SPME}}}
\newcommand{\rms}{\mathrm{rms}}
\newcommand{\diff}{\mathrm{d}}

\newtheorem{theorem}{Theorem}
\newtheorem{remark}{Remark}

\newcommand{\etal}{\textit{et al.\ }}

\usepackage{etoolbox}
\AtBeginEnvironment{align}{\setcounter{subeqn}{0}}
\newcounter{subeqn} %

\begin{document}
\begin{frontmatter}
\title{A comparison of the Spectral Ewald and Smooth Particle Mesh Ewald methods in GROMACS}
\author[aut]{Davood Saffar Shamshirgar}
\ead{davoudss@kth.se}

\author[autf]{Berk Hess}
\ead{hess@kth.se}

\author[aut]{Anna-Karin Tornberg}
\ead{akto@kth.se}

\address[aut]{KTH Mathematics, Swedish e-Science
  Research Centre, 100 44
  Stockholm, Sweden.}
\address[autf]{KTH Computational Physics, Swedish e-Science
  Research Centre, 106 91
  Stockholm, Sweden.}


\begin{abstract}
The smooth particle mesh Ewald (SPME) method is an FFT based method
for the fast evaluation of electrostatic interactions under periodic
boundary conditions. A highly optimized implementation of this method
is available in GROMACS, a widely used software for molecular
dynamics simulations. In this article, we compare a more recent method
from the same family of methods, the spectral Ewald (SE) method, to
the SPME method in terms of performance and efficiency. We consider
serial and parallel implementations of both methods for single and
multiple core computations on a desktop machine as well as the Beskow
supercomputer at KTH Royal Institute of Technology. The
implementation of the SE method has been well optimized, however not
yet comparable to the level of the SPME implementation that has been
improved upon for many years.  We show that the SE method
is very efficient whenever used to achieve high accuracy and that it
already at this level of optimization can  be competitive for low
accuracy demands.
\end{abstract}


\end{frontmatter}

\section{Introduction}
\label{section:introduction}
In molecular dynamics simulations (MD), the most time consuming task
is often the computation of the electrostatic interactions. This remains
true despite the fact that many algorithms exist that can reduce the
computational cost from $\order{N^2}$ to $\order{N \log N}$ for evaluating the
interaction of $N$ particles under the common assumption of periodic
boundary conditions. Therefore, this can not be considered a solved problem, and
improvements are highly desirable.

The problem can be stated as follows. Suppose $N$ charged particles with charges $q_{\ni}$ are located in positions $\v x_{\ni}$, $\ni=1,\ldots,N$, and interacting in a simulation box $\Omega=[0,L)^3$. Moreover, assume the so-called \textit{charge neutrality} condition, $\sum_{\ni=1}^N q_{\ni}=0$. This condition is added to avoid the divergence of the electric potential energy in the presence of periodic boundary conditions. The objective is to compute the electrostatic potential,
\begin{equation}
\varphi(\v x_{\mi})=\sum_{\v p}^{'}\sum_{\ni=1}^N\frac{q_{\ni}}{|\v x_{\mnp}|},\quad \v x_{\mnp}=\v x_{\mi\ni}+\v p L=\v x_{\mi}-\v x_{\ni}+\v p L,\quad\v p\in\mathbb{Z}^3,
\label{eq:intro:potential}
\end{equation}
or two related quantities, the electrostatic potential energy,
\begin{equation}
E=\frac{1}{2}\sum_{\mi=1}^Nq_{\mi}\varphi(\v x_{\mi}),
\label{eq:intro:energy}
\end{equation}
and the electrostatic force,
\begin{equation}
\v F(\v x_{\mi})=-\frac{\p E}{\p \v x_{\mi}}=-\frac{1}{2}q_{\mi}\frac{\p \varphi(\v x_{\mi})}{\p \bx_{\mi}},
\label{eq:intro:force}
\end{equation}
$\mi=1,\ldots,N$. The sum over $\v p$ is over all periodic images of particles and the prime ($'$) denotes that the term with $\lbrace \mi=\ni,\v p=0\rbrace$ is excluded from the sum. The electrostatic potential \eqref{eq:intro:potential} decays as $1/r$ (inverse of the distance between particles) and is very slow to compute in this form. Moreover, the sum is only conditionally convergent and the result depends on the order of summation. If the sum is evaluated in a spherical order, it yields the same result as the Ewald potential \cite{smith}.
The \textit{Ewald summation} technique proposed by Ewald \cite{ewald} in 1921 can be used to compute the conditionally convergent sums of \eqref{eq:intro:potential}-\eqref{eq:intro:force} by decomposition of $1/r$ into two parts as 
\begin{equation}
\frac{1}{r}=\frac{\erfc(\xi r)}{r} + \frac{\erf(\xi r)}{r},
\label{eq:intro:decomposition}
\end{equation}
where $\erf(\cdot)$ and $\erfc(\cdot)$ are the \textit{error function} and the \textit{complementary error function} respectively and $\xi>0$ is called the \textit{Ewald parameter}. The first term in \eqref{eq:intro:decomposition} tends to zero quickly as $r\to\infty$. The second term, on the other hand, is a smoothly varying function with a rapidly converging representation in Fourier space.
Using the decomposition \eqref{eq:intro:decomposition} we can rewrite \eqref{eq:intro:potential} and \eqref{eq:intro:energy} as
\begin{align}
\varphi(\v x_{\mi}) =& \varphi^\S_{\mi}+\varphi^\L_{\mi}+\varphi^{\mathrm{self}}_{\mi} \nonumber \\
=&\sum_{\ni}\sum_{\v p}^{'}q_{\ni}\frac{\erfc(\xi |\v x_{\mnp}|)}{|\bx_{\mnp}|}+\frac{4\pi}{L^3}\sum_{\v k\neq0}\frac{e^{-k^2/4\xi^2}}{k^2}\sum_{\ni}q_{\ni}e^{\ii\v k\cdot\v x_{\mn}}-\frac{2\xi}{\sqrt{\pi}}q_{\mi},
\label{eq:intro:potential_ewald}\\
E =& E^\S+E^\L+E^{\mathrm{self}} \nonumber \\
=&\frac{1}{2}\sum_{\mi}\sum_{\ni}\sum_{\v p}^{'}q_{\mi}q_{\ni}\frac{\erfc(\xi |\v x_{\mnp}|)}{|\v x_{\mnp}|} \nonumber \\
&+\frac{4\pi}{2L^3}\sum_{\v k\neq0}\frac{e^{-k^2/4\xi^2}}{k^2}\sum_{\mi}\sum_{\ni}q_{\mi}q_{\ni}e^{\ii\v k\cdot\v x_{\mn}} 
-\frac{\xi}{\sqrt{\pi}}\sum_{\mi}q_{\mi}^2.
\label{eq:intro:energy_ewald}
\end{align}
The electrostatic force can also be obtained by differentiating the energy equation \eqref{eq:intro:energy_ewald} with respect to $\v x_\mi$,
\begin{align}
\v F(\v x_{\mi})=&F^\S+F^\L \nonumber \\
=&q_{\mi}\sum_{\ni}q_{\ni}\sum_{\v p}^{'}\left(\frac{2\xi}{\sqrt{\pi}}e^{-\xi^2|\v x_{\mnp}|^2}+\frac{\erfc(\xi|\v x_{\mnp}|)}{|\v x_{\mnp}|}\right)\frac{\v x_{\mnp}}{|\bx_{\mnp}|^2} \hspace{2.5cm} \nonumber \\
&-\frac{4\pi \ii q_{\mi}}{2L^3}\sum_{\v k\neq0}\frac{\v k e^{-k^2/4\xi^2}}{k^2}\sum_{\ni}q_{\ni} e^{\ii\v k\cdot \v x_{\mn}},
\label{eq:intro:force_ewald} 
\end{align}
where $\v k\in\lbrace 2\pi \v n/L,\v n\in\mathbb{Z}^3\rbrace$ and $k=|\v k|$. In \eqref{eq:intro:potential_ewald}-\eqref{eq:intro:force_ewald}, superscripts \textit{S}, \textit{L} and \textit{self} refer to \textit{Short range} (computed in real space), \textit{Long range} (computed in Fourier space) and \textit{Self interaction contribution} respectively. The decomposition parameter $\xi$, controls the convergence of the two sums in a way that larger $\xi$ results in a faster convergence of the real space sum in expense of slow convergence of the Fourier space sum. Furthermore, equations \eqref{eq:intro:potential_ewald}-\eqref{eq:intro:force_ewald} are exact in the current form and their results are invariant to the Ewald parameter.
In a charge neutral system, the mean-value of the potential will be zero and, therefore, the zero mode of the Fourier transformed potential is set to zero, i.e.\ $\wh{\varphi}(\v k=0)=\langle\varphi^L(\v x)\rangle=0$. 

In the formulas above, the real space sum can be evaluated ignoring
interactions between particles further apart than a given distance 
such that the $\erfc(\cdot)$ function is well enough decayed. The Fourier space sum can be evaluated using various methods that differ in structure. Among all methods proposed for the Fourier space sum, are the fast Fourier transform (FFT) based methods \cite{pme,p3m,spme,hockneyeastwood,se}, non-FFT based or Hierarchical methods \cite{fmm} and PDE-based methods \cite{pdebased}. These methods are all widely used in major software packages such as GROMACS \cite{gromacs}, AMBER \cite{amber}, NAMD \cite{namd}, and SCAFACOS \cite{scafacos}. 

The aim of FFT-based methods is to use the FFT in order to accelerate
 the calculation of the Fourier space sum. This allows for a different
 choice of the Ewald parameter to shift more work to the Fourier part,
 and decrease it in the real space calculations. With a proper scaling,
 these methods reduce the total complexity of the summation to $\order{N \log N}$.

Among the most important methods which benefit from the FFT, are methods within the Particle-Mesh Ewald (PME) family, including the original PME method by Darden \etal \cite{pme}, the SPME method by Essmann \etal \cite{spme}, Particle-Particle-Particle-Mesh Ewald (P$^3$M) originally introduced by Hockney and Eastwood \cite{hockneyeastwood}, its variant Interlaced P$^3$M by Neelov and Holm \cite{interlaced}, and Particle-Mesh NFFT by Pippig and Potts \cite{pnfft}. In this paper, we review another method of the PME family, the Spectral Ewald (SE) method introduced by Lindbo $\&$ Tornberg \cite{se}, and compare it to the SPME method. 

We provide a comparison between the spectral
Ewald (SE) and smooth particle mesh Ewald (SPME) methods in GROMACS
for different systems of particles under three-dimensional periodic
boundary conditions. We discuss similarities and differences,
advantages and disadvantages for the two methods, including the
approximation errors that they introduce and the computational cost
that is incurred for different parts of the algorithms.

The evaluation of the real space sum is the same in the two methods. Therefore, 
we only study the runtime estimate of this sum and do not discuss the 
details of the algorithm used for this calculation. 
In the evaluation of the $k$-space sum, a regular FFT grid is
introduced. The difference between the two methods lies in the choice
of window functions that are used to
interpolate the point charges onto this grid in the first step of the
algorithm, as well as to interpolate back to the irregular point
locations from grid values at the end of the calculations - the other
differences follow as a consequence of this choice.

In the SE method, the window function is a
suitably scaled and truncated Gaussian. Given the number of points in
the support of the Gaussian and the grid size, parameters are adjusted
to scale it optimally. This way, approximation errors introduced in
the SE method are essentially independent of the grid size, and the
size of the FFT grid can be chosen considering only the truncation of
the original Ewald $k$-space sum.

In contrast, the SPME method works with a $p$th order cardinal B-spline
as the window function, which introduces an
approximation error of order $h^p$, where $h$ is the grid size of the FFT
grid. Hence, the size of the FFT grid must be in general larger than the 
truncation of the $k$-space sum (more points, smaller $h$) to reduce the 
approximation errors.

In terms of computational cost to obtain a certain accuracy, the cost
from computing FFTs will hence be smaller for the SE method as
compared to the SPME method, since smaller grids will be used. On the
other hand, the number of points in the support of the Gaussian is
larger than that of the support of the cardinal B-splines, and hence
the cost of interpolating to and from the grid will be larger for the
SE method. In terms of which method that offers the lowest
computational cost, this will depend on systems under study and parameters, as well
as the quality of the implementation. Generally speaking, the SE
method will outperform the SPME method if high accuracy is required,
whereas the SPME method will be more efficient if only a few digits of
accuracy is asked for.

We have implemented the SE method in GROMACS ver 5.1 \cite{gromacs}, using
the same routines for the real space sum and for evaluating FFTs
as the SPME method. The remaining implementation of the SE method has
been well optimized, even if not as fine tuned as the implementation
for the SPME method that has been available in GROMACS since late 1999. Hence,
the computational runtimes for SE relative to SPME can be expected to be
reduced by additional optimization of the implementation.

This article is organized as follows. In section \ref{section:truncation} we review
truncation error estimates in Ewald sums. The structure of the
FFT-based methods is reviewed in section \ref{section:fftbased}. Then, in section \ref{section:se}, we
review the SE method and consider errors in the SE and SPME methods
\cite{spme} in section \ref{section:error_se} and \ref{section:error_spme} respectively. The reader that is familiar with
the SPME method can get a quick overview of the differences in
implementation between the two methods by glancing at table \ref{tab:compare} in
section \ref{subsection:methodology}. In section \ref{section:runtime_estimate} we discuss runtime estimates of the SE
method in GROMACS. The numerical results for the serial and parallel
implementations are presented in sections \ref{section:serial} and  \ref{section:parallel}. 

\section{Ewald summation method}
\subsection{Truncation error estimates}
\label{section:truncation}
The sums in \eqref{eq:intro:potential_ewald} are all infinite and need to be truncated in practice. The error committed due to this truncation is discussed in this section. To consider one measure of accuracy, we use the root mean square (rms) error,
\begin{align*}
e_{\rms}^2:=\dfrac{1}{N}\sum_{\ni=1}^N (\varphi(\v x_{\ni})-\varphi^\ast(\v x_{\ni}))^2,
\end{align*}
where $\varphi^\ast$ denotes an exact or a well converged reference solution. Choosing a cut-off radius $r_\c<L/2$, the real space sum in \eqref{eq:intro:potential_ewald}-\eqref{eq:intro:force_ewald} is truncated such that the sum only includes the terms for which $|\v x_{\mnp}|\leq r_\c$. Using the well established Kolafa $\&$ Perram error estimates \cite{kolafa}, the real space truncation errors are estimated as
\begin{subequations}
\label{eq:truncation:kolafa_real}
\begin{align}
e_{\P,\rms}^\S &\approx (Qr_\c/2L^3)^{1/2}(\xi r_\c)^{-2}e^{-r_\c^2\xi^2},\\
e_{\E,\rms}^\S &\approx Q(r_\c/2L^3)^{1/2}(\xi r_\c)^{-2}e^{-r_\c^2\xi^2},\\
e_{\F,\rms}^\S &\approx 2Q(1/r_\c L^3)^{1/2}e^{-r_\c^2\xi^2},
\end{align}
\end{subequations}
for the potential, energy and force respectively and with $Q:=\sum_{\ni=1}^Nq_{\ni}^2$ . The Fourier space sum includes an exponential term of the form $e^{-k^2/4\xi^2}$ which decays exponentially as $k$ increases. This sum can also be truncated at some maximum wave number $k_{\infty}\in\mathbb{Z}^+$ such that it only includes the terms for which $|\v k|\leq 2\pi k_{\infty}/L$. The Fourier space truncation errors are also estimated as \cite{kolafa},
\begin{subequations}
\label{eq:truncation:kolafa_fourier}
\begin{align}
e_{\P,\rms}^\L &\approx \xi\pi^{-2}k_{\infty}^{-3/2}\sqrt{Q}e^{-(\pi k_{\infty}/\xi L)^2}, \\
e_{\E,\rms}^\L &\approx \xi\pi^{-2}k_{\infty}^{-3/2}Q e^{-(\pi k_{\infty}/\xi L)^2}, \\
e_{\F,\rms}^\L &\approx \xi(L \pi)^{-1}(8/k_{\infty})^{1/2}Q e^{-(\pi k_{\infty}/\xi L)^2}.
\end{align}
\end{subequations}
Given a certain error tolerance $\e$ and a cut-off radius $r_\c$, the Ewald parameter $\xi$ can be obtained using \eqref{eq:truncation:kolafa_real}
\begin{subequations}
\begin{align}
\xi_\P &=\frac{1}{r_\c}\left[ W\left(\frac{1}{\e}\sqrt{\frac{Qr_\c}{2L^3}}\right)\right]^{1/2}, \\
\xi_\E &=\frac{1}{r_\c}\left[ W\left(\frac{1}{\e}Q\sqrt{\frac{r_\c}{2L^3}}\right)\right]^{1/2}, \\
\xi_\F &=\frac{1}{r_\c}\left[\log\left(\frac{2Q}{\e}\sqrt{\frac{1}{r_\c L^3}}\right)\right]^{1/2}.
\end{align}
\label{eq:truncation:kolafa_xi}%
\end{subequations}
Here, $W(\cdot)$ denotes the \textit{Lambertw} function and is defined as the inverse of $f(x)=xe^{-x}$. Having $\xi$ in hand we find $k_{\infty}$ through
\begin{subequations}
\begin{align}
k_{\P,\infty} &=\frac{\sqrt{3}L\xi}{2\pi}\left[ W\left( \frac{4}{3L^2}(\frac{Q}{\pi\xi\e^2})^{2/3} \right)\right]^{1/2}, \\
k_{\E,\infty} &=\frac{\sqrt{3}L\xi}{2\pi}\left[ W\left( \frac{4}{3L^2}(\frac{Q^2}{\pi\xi\e^2})^{2/3} \right)\right]^{1/2}, \\
k_{\F,\infty} &=\frac{L\xi}{2\pi} \left[ W\left( \frac{2^8\xi^2Q^4}{\e^4L^6\pi^2} \right)\right]^{1/2},
\end{align}
\label{eq:truncation:kolafa_k_inf}%
\end{subequations}
for the potential, energy and force. For a fixed cut-off $r_\c$ and an error tolerance $\e$, we have $\xi_\P\leq \xi_\E \leq \xi_\F$ as given by
(\ref{eq:truncation:kolafa_xi}a-c). Based on the relations in (\ref{eq:truncation:kolafa_k_inf}a-c) and for given values of $\xi$ and $\e$, we have that, $k_{\P,\infty}$ is the smallest among all three and $k_{\E,\infty}$ and $k_{\F,\infty}$ are comparable.

As was mentioned before, the Ewald parameter controls the relative decay of the real space and Fourier space sums. For a fixed error tolerance $\e$, allowing for a larger $r_\c$ will yield a smaller $\xi$, and in turn, a smaller $k_{\infty}$, decreasing the computational cost of the Fourier space sum. The optimal $\xi$ is chosen such that the computational cost of evaluating the real and Fourier space sums are balanced. This however depends on the algorithm, implementation and hardware used for simulations.

\subsection{Computational complexity}
\label{section:complexity}
The direct evaluation of the Ewald sums \eqref{eq:intro:potential_ewald}-\eqref{eq:intro:force_ewald}, has a complexity of $\order{N^2}$, where $N$ is the number of particles in the system. With a proper scaling of $\xi$ with $N$, and assuming a uniform distribution of charges, the computational cost of the Ewald sum can be reduced to ${\cal O}(N^{3/2})$ \cite{perram}, though, with a large constant. In the next section we will consider methods that reduce the complexity to $\order{N\log{N}}$, also with a much smaller constant.

\section{FFT-based methods}
\label{section:fftbased}
As mentioned earlier, the aim of FFT-based methods is to use the FFT
in order to accelerate the calculation of the Fourier space sum. These
methods reduce the total complexity of the summation from $\order{N^2}$ (or
$\order{N^{3/2}}$ by optimizing the involved parameters) to $\order{N\log{N}}$ with a
proper scaling of the Ewald parameter. The FFT acceleration allows for a
larger $\xi$, corresponding to a smaller $r_\c$, and thereby reduces the cost
of evaluating the real space sum. We shall get back to this discussion
later in section \ref{section:se:complexity} when we present the computational complexity of the SE method.

The principles of all FFT-based methods within the Particle-Mesh Ewald
(PME) family, including the SPME and SE methods, are the same. All methods are inherited from the original P$^3$M and PME methods. In all these methods, the evaluation is conducted in Fourier space by some modification to make possible the use of the FFT efficiently. Considering the evaluation of the Fourier space part of the potential ($\varphi^\L_{\mi}$ in \eqref{eq:intro:potential_ewald}), the resulting algorithms have the following general steps,

{
\small
\begin{breakablealgorithm}
\caption{The structure of FFT-based Ewald summation methods}
\small
\label{alg:fftbased}
\begin{algorithmic}[1]
\State A uniform grid is introduced.
\State Spreading step: Using a suitable window function, the charges of the non-uniform particles are distributed on the uniform grid.
\State A 3D FFT is performed on the grid function.
\State The grid function is scaled by a multiplication with a modified Fourier transformed Green's function in Fourier space.
\State A 3D inverse FFT is performed.
\State Gathering step: The potential (force or energy) is evaluated at targets with the same window function as in the spreading step.
\end{algorithmic}
\end{breakablealgorithm}
}

The spreading and gathering steps can be performed using different window functions, resulting in different flavors of methods with different accuracies and computational costs. The modified Green's function in step 4 depends on the choice of the window function in steps 2 and 6. In all FFT-based methods, the Fourier space truncation error is determined by the number of grid points as given by the Kolafa $\&$ Perram error estimates, but the approximation errors depend on the details of the methods. We discuss more on the approximation errors of the SE and SPME methods in sections \ref{section:error_se:approximation_error} and \ref{section:error_spme:approximation_error}.

\subsection{Force calculation}
\label{subsec:force_calc}
Referring back to \eqref{eq:intro:force}, the calculation of the electrostatic force involves differentiation of the electrostatic energy. The differentiation can be performed in three different ways which differ in accuracy and computational cost. In the first approach, used in the original PME method, differentiation is done with $\ii\v k$ multiplication in Fourier space. This is the most accurate approach to evaluate the force but it is computationally very expensive as it needs two extra 3D IFFTs. The second way of obtaining the force, used in the original P$^3$M method, is to perform a discrete differentiation on the mesh in real space. This method, however, is only useful for low accuracy demands. The last approach is to analytically differentiate the charge assignment function in real space. This method is employed in the SPME method and since it only requires one 3D IFFT, it is cheaper to perform as compared to the $\ii \v k$ differentiation. Another important advantage of the latter approach is that, as a result of analytic differentiation of the energy, and independent of the truncation errors, the resulting scheme conserves energy. This, in addition to the performance, is an important reason for us to choose this scheme for the SE method as well (see section \ref{section:se:force}).

\section{The Spectral Ewald method}
\label{section:se}
Here, we review the Spectral Ewald (SE) method \cite{se} for electrostatic calculations
under fully periodic boundary conditions. As mentioned before, the method has
the structure of all PME family methods but unlike other methods in the family, it is
spectrally accurate. Consider the evaluation of the long range part of
the electrostatic potential in \eqref{eq:intro:potential_ewald},
\begin{align}
\varphi^\L(\v x_{\mi})=\frac{4\pi}{L^3}\sum_{\v k\neq0}\frac{e^{-k^2/4\xi^2}}{k^2}\sum_{\ni}q_{\ni}e^{\ii\v k\cdot\v x_{\mn}},
\label{eq:se:potential}
\end{align}
under the assumptions made in section \ref{section:introduction}. In the SE method, suitably scaled (and truncated) Gaussian functions are used in the spreading
and gathering steps of Algorithm 1.  To derive the relevant formulas,
a free parameter $\eta\in(0,1)$ is introduced to split the Gaussian factor
$e^{-k^2/4\xi^2}$ as,
\begin{align*}
e^{-k^2/4\xi^2}=\underbrace{e^{-\eta k^2/8\xi^2}}_a\underbrace{e^{-(1-\eta)k^2/4\xi^2}}_b\underbrace{e^{-\eta k^2/8\xi^2}}_c.
\end{align*}
The term $(a)$ will be used in the spreading step, $(b)$ for multiplication in Fourier space and $(c)$ in the gathering step. We start by writing \eqref{eq:se:potential} as,
\begin{align}
\varphi^\L(\v x_{\mi})&=\frac{4\pi}{L^3}\sum_{\v k\neq0}\frac{e^{-k^2/4\xi^2}}{k^2}e^{\ii\v k\cdot\v x_{\mi}}\sum_{\ni}q_{\ni}e^{-\ii\v k\cdot\v x_{\ni}} \nonumber \\
&=\frac{4\pi}{V}\sum_{\v k\neq0}\frac{e^{-(1-{\eta})k^2/4\xi^2}}{k^2}e^{\ii\v k\cdot\v x_{\mi}}e^{-{\eta} k^2/{8}\xi^2}\sum_{\ni}q_{\ni}e^{-{\eta} k^2/{8}\xi^2}e^ {-\ii\v k\cdot\v x_{\ni}}.
\label{eq:se:potential_split}
\end{align}
Now, \eqref{eq:se:potential_split} can be written as
\begin{align}
\varphi^\L(\v x_{\mi})=\frac{4\pi}{L^3}\sum_{\v k\neq0}\frac{e^{-(1-\eta)k^2/4\xi^2}}{k^2}e^ {\ii\v k\cdot\v x_{\mi}}e^{-\eta k^2/8\xi^2}\wh{H}_{-k},
\label{eq:se:potential_split_Hhat}
\end{align}
where
\begin{align}
\wh{H}_k:=\sum_{\ni}q_{\ni}e^{-\eta k^2/ 8\xi^2}e^{\ii\v k\cdot\v x_{\ni}}.
\label{eq:se:Hhat}
\end{align}
The inverse Fourier transform of \eqref{eq:se:Hhat} is given by
\begin{align}
H(\v x)=\sum_{\ni}q_{\ni}\left(\frac{2\xi^2}{\pi\eta}\right)^{3/2}e^{-2\xi^2|\v x-\v x_{\ni}|_{\ast}^2/\eta},
\label{eq:se:grid}
\end{align}
where $^{\ast}$ denotes that periodicity is applied in all directions. The function $H(\v x)$ is a smooth and periodic function, evaluated on a uniform grid. Algorithmically, evaluating $H(\v x)$ corresponds to the spreading step, step 2 in Algorithm 1. Since the charges are distributed on a uniform grid, an FFT can be applied to compute $\wh{H}_k$ efficiently. The Fourier transformed smeared particle function $\wh{H}_k$, are then scaled by $k^{-2}e^{-(1-\eta)k^2/4\xi^2}$ (step 4 in Algorithm 1). We define
\begin{align}
\wh{\wt{H}}_k:=\frac{e^{-(1-\eta)k^2/4\xi^2}}{k^2}\wh{H}_{k},
\label{eq:se:scaling}
\end{align}
and with this, \eqref{eq:se:potential_split_Hhat} is simplified to
\begin{align}
\varphi^\L(\v x_{\mi})=\frac{4\pi}{L^3}\sum_{\v k\neq0}\wh{\wt{H}}_{-k}e^{-\eta k^2/8\xi^2}e^{\ii\v k\cdot\v x_{\mi}}.
\label{eq:se:Htildehat}
\end{align}
Then, as a result of exerting the Parseval theorem on \eqref{eq:se:Htildehat}, one obtains
\begin{align}
\varphi^\L(\v x_{\mi})=4\pi\left(\frac{2\xi^2}{\pi\eta}\right)^{3/2}\int_{\Omega} \wt{H}(\v x)e^{-2\xi^2|\v x-\v x_{\mi}|_{\ast}^2/\eta}\diff\v x.
\label{eq:se:potential_integral}
\end{align}
This interpolation, corresponds to the gathering step, step 6 in Algorithm 1. Finally, the integral in \eqref{eq:se:potential_integral} can be evaluated by the trapezoidal rule,
\begin{align}
\varphi^\L_\SE(\v x_{\mi})= 4\pi h^3\left(\frac{2\xi^2}{\pi\eta}\right)^{3/2}\sum_{\ni} \wt{H}(\v x_{\ni})e^{-2\xi^2|\v x_{\ni}-\v x_{\mi}|_{\ast}^2/\eta}.
\label{eq:se:potential_discrete}
\end{align}
Due to the fact that $\wt{H}(\v x)$ is smooth and periodic (as a result of periodicity and smoothness of $H$), the trapezoidal rule yields spectral accuracy. Moreover, $\wt{H}(\v x)$ in \eqref{eq:se:potential_discrete} can be evaluated by applying an inverse FFT on $\wh{\wt{H}}$. Note here that, the same interpolation function used in the spreading \eqref{eq:se:grid} and gathering \eqref{eq:se:potential_integral} steps. The reader may consult \cite{se} for more details on the derivation of the SE method. The formulas that will be used in the different steps of Algorithm 1
are summarized in table \ref{tab:compare} of section \ref{subsection:methodology}, side by side with the
formulas needed for the SPME method.

Consider a uniform grid with $M$ points in each direction, where $M$ is related to the maximum Fourier mode $k_{\infty}$ by $M=2k_{\infty}$. The Gaussians do not have compact support, and hence they have to be truncated. Let $P$, where $P\leq M$, be the number of points in the support of Gaussians in each direction. Also we denote by $w=Ph/2=PL/2M$ the half width of the Gaussians. Earlier in this section, we introduced the parameter $\eta$ which controls the width of Gaussians. At the point of truncating the Gaussians, $|\v x-\v x_{\ni}|_{\ast}=w$, assume that the term $e^{-2\xi^2|\v x-\v x_{\ni}|_{\ast}^2/\eta}$ has decayed to $e^{-m^2/2}$. Hence, the choice of $m$ will completely determine $\eta$ and control the level at which the Gaussians are truncated. Then,
\begin{align*}
\eta=\left(\frac{2w\xi}{m}\right)^2.
\end{align*}
In section \ref{section:error_se:approximation_error} we will explain how the choice of $m$ is tied to $P$ in order to find the optimal $\eta$. As a consequence, $P$ remains the only parameter to select.

Truncating the Gaussians decreases the computational complexity of the spreading and gathering steps significantly, but the computations in these steps are still expensive as they involve evaluation of $P^3$ exponentials for each particle. Using the Fast Gaussian Gridding (FGG) approach suggested by Greengard $\&$ Lee \cite{greengardlee}, these computations are accelerated at the expense of using more memory space to store parts of the exponentials and reuse them. Here, we briefly describe the idea of the FGG and refer the reader to \cite{greengardlee} or \cite{se} for more details. Let $\v x_{\ni}$ be a particle position and $\v x=(i,j,k)h$, $(i,j,k)\in\mathbb{Z}^3$ a grid point. Then, the evaluation of 
\begin{align}
e^{-c(\v x-\v x_{\ni})^2}=e^{-c(ih-x_{\ni})^2}e^{-c(jh-y_{\ni})^2}e^{-c(kh-z_{\ni})^2},
\label{eq:fgg}
\end{align}
involves $M^3N$ exponential evaluations. Using the FGG approach we can rewrite \eqref{eq:fgg} e.g.\ in the $x$-direction as,
\begin{align}
e^{-c(ih-x_{\ni})^2}=\underbrace{e^{-c(ih)^2}}_a\underbrace{(e^{2chx_{\ni}})^ie^{-cx_{\ni}^2}}_b.
\end{align}
Since the term $(a)$ is independent of particle positions, it can be precomputed and stored. This computation on the whole grid involves evaluation of $M^3$ exponentials and requires $M^3$ storage. The term $(b)$, on the other hand, can be precomputed once for each particle and requires evaluation of $4N$ exponentials, $3(2N+MN)$ multiplications and $3NM$ storage in three dimensions. Note that the first term in $(b)$ is computed by successive multiplications with itself. The FGG approach is much more effective as compared to the direct approach when the Gaussians are truncated with $P$ points in the support in each direction. Table \ref{tab:se:fgg} shows that using the FGG, the cost of evaluating \eqref{eq:se:grid} and \eqref{eq:se:potential_discrete} reduces from $\order{P^3N}$ to $\order{PN}$ with the expense of $\order{PN}$ multiplication and storage.
\begin{table}[htbp]
\begin{center}
\begin{tabular}{c|cc}
& \textbf{Direct approach} & \textbf{FGG approach} \\ 
\hline
\textbf{\textbf{exp($\cdot$)} evaluation} & $P^3N$ & $P^3+3PN+4N$\\ 
\textbf{multiplication} & - & $3(2N+PN)$ \\
\textbf{storage} & - & $P^3+3PN$ \\
\end{tabular} 
\end{center}
\caption{The cost of evaluating exponentials in \eqref{eq:se:grid} and \eqref{eq:se:potential_discrete}.}
\label{tab:se:fgg}
\end{table}

\subsection{Calculation of the electrostatic force}
\label{section:se:force}
Based on the conclusion made in section \ref{subsec:force_calc}, to compute the force in the SE method, we differentiate the spreading function. The charge spreading function in the SE method has the form of $e^{-2\xi^2|\v x-\v x_{\ni}|_{\ast}^2/\eta}$. By analytic differentiation of this function and using \eqref{eq:intro:force}, \eqref{eq:se:potential_integral} is then reformulated as,
\begin{align}
\v F(\v x_{\mi})= 2\pi \frac{4\xi^2}{\eta}\left(\frac{2\xi^2}{\pi\eta}\right)^{3/2} q_{\mi} \int_{\Omega}(\v x-\v x_{\mi})_{\ast}\wt{H}(\v x)e^{-2\xi^2|\v x-\v x_{\mi}|_{\ast}^2/\eta}\diff\v x.
\label{eq:se:force}
\end{align}
As in \eqref{eq:se:potential_discrete}, this integral can be approximated by the trapezoidal rule to spectral accuracy. We remark here that since the force and potential differ only in the term $\v x-\v x_{\mi}$, it is possible to compute these quantities simultaneously at the interpolation step.

\subsection{Computational complexity of the SE method} 
\label{section:se:complexity}
The computational cost of the SE method can be divided into three parts; the spreading and gathering cost, the scaling cost, and the FFT cost. Consider a system of $N$ particles located in a box of length $L$ and assume that the system particle density $\rho=N/L^3$ is constant as $N$ grows, i.e.\ $L^3\propto N$ or equivalently $L\propto N^{1/3}$. We assume again that the domain $[0,L)^3$ is discretized with a total of $M^3$ points and truncated Gaussians have $P^3$ points in their supports. Therefore, the computational complexity of the spreading and gathering steps are $\order{NP^3}$. The scaling step involves $M^3$ multiplications and, therefore, the cost of this step scales as $M^3$. To evaluate the full complexity of the algorithm, and tie $N$ to $M$, we need to account also for the computational complexity of the real space part of the Ewald sum. Let us assume that the cut-off radius $r_\c$ is fixed. As $N$ grows, the real space sum scales as $N$ with a somewhat large constant that depends on the average number of nearest neighbors. Let $\tau$ be the level of accuracy at which we truncate the real and Fourier space sums. Since $r_\c$ is fixed, the error estimate \eqref{eq:truncation:kolafa_real} suggests that $e^{-\xi^2r_\c^2}\approx\tau$. This implies that $\xi$ is fixed. Then, from the error estimate \eqref{eq:truncation:kolafa_fourier}, $e^{-\pi k_{\infty}^2/(\xi L)^2}\approx\tau$ and therefore, $M\propto k_{\infty}\propto L$, i.e.\ $M^3\propto N$. Moreover, if $h=L/M$ is fixed, there exists an integer $P$ such that the approximation error also satisfies the requirement on the level of accuracy. Hence, the cost of the Fourier space scaling and FFTs are of the order $N$ and $N\log N$ respectively, and the full algorithm has a complexity of $\order{N\log N}$.

\section{Errors in the SE method}
\label{section:error_se}
In addition to the truncation errors that stem from truncating the sums with infinite number of terms, there are also approximation errors introduced in the SE method. These errors occur due to truncation of the Gaussians and using quadrature rules to compute the integrals in \eqref{eq:se:potential_integral} and \eqref{eq:se:force}, for the potential and force, respectively. Here, we aim to present experimental upper bounds for the approximation errors in the evaluation of the potential, energy and force and provide numerical results to assess the sharpness of the error bounds.

\subsection{Approximation error in the potential  calculation}
\label{section:error_se:approximation_error}
The derivation of \eqref{eq:se:potential_integral} (or \eqref{eq:se:force}) involves no approximation, but we commit errors when we evaluate the integral with the trapezoidal rule and also truncate the Gaussians. The following theorem adapted from \cite{se} provides an upper bound for the approximation error in evaluating the potential and sets a foundation for obtaining corresponding bounds for the energy and force.

\begin{theorem} (Error estimate in the potential calculation using the SE method).
Given $\xi$, $h>0$ and an integer $P>0$, define $w=hP/2$ and $\eta=(2w\xi/m)^2$. The error committed in evaluating the Fourier space part of the potential \eqref{eq:se:potential} by truncating the Gaussians at $|\normalfont\v x-\v x_\mi|=w$ and applying the trapezoidal rule as in \eqref{eq:se:potential_discrete} can be estimated by
\begin{align}
|\varphi_\SE^\L-\varphi^\L|_{\rms}\leq|\varphi_\SE^\L-\varphi^\L|_{\infty}\leq C(e^{-\pi^2P^2/(2m^2)}+\emph{\erfc}(m/\sqrt{2})),
\label{eq:se:approx_error_bound_potential}
\end{align}
and $C$ is a constant independent of $m$ and $P$.
\label{theo:approx_error}
\end{theorem}
\begin{proof}
See \cite{se}.
\end{proof}
Clearly, the error estimate is related to $m$, the shape controller, and $P$. It is natural to balance two terms such that they contribute equally in the error. This gives, $m=c\sqrt{\pi P}$ and we experimentally find $c=0.95$ to set a good choice and hence
\begin{align}
\eta=\left(\dfrac{2w\xi}{m}\right)^2=\dfrac{PL^2\xi^2}{c^2\pi M^2}.
\label{eq:eta}
\end{align}
Therefore, we can write 
\begin{align}
|\varphi_\SE^\L-\varphi^\L|_{\rms}\lessapprox A_\P e^{-\pi Pc^2/2}.
\end{align}
We emphasize here that, now, $P$ is the only parameter to control the approximation error and in contrast to the other PME-like methods, the approximation error is independent of the grid size, $M$. To find $A_\P$, we choose the grid size large enough for the truncation error to be negligible in the Fourier space sum. Through numerical experiments, we find that $A_\P\approx \sqrt{Q\xi L}/L$. To confirm this, the error in the Fourier space part of the potential has been measured for 100 different cases and for each $P$ including all possible combinations of $L\in\lbrace1,5,10,20,40\rbrace$, $N\in\lbrace200,400,800,1200\rbrace$, and $\xi\in\lbrace5,15,25,30,35\rbrace/L$. In figure \ref{fig:se:approx_error_bound} (left), we plot absolute rms error in the potential calculation as scaled by $A_\P$. We observe that the error for 100 different cases and all $P$ collapse rather well under this scaling. It also turns out that if $P$ is more than 24, independent of the grid size, machine accuracy is attained, \cite{se}.

\subsection{Approximation error in the energy  calculation}
Following a similar approach as in Theorem \ref{theo:approx_error}, the approximation error in the energy can be written as,
\begin{align}
|E^\L_\SE-E^\L|_{\rms}\lessapprox A_\E e^{-\pi Pc^2/2},
\label{eq:se:approx_error_bound_energy}
\end{align}
and considering \eqref{eq:intro:energy_ewald}, the natural choice would be to choose $A_\E\approx\frac{Q\sqrt{\xi L}}{L}$.

\subsection{Approximation error in the force calculation}
\label{section:error_se:approximation_error_force}
It is straightforward to do the similar procedure as in \cite{se} to derive the approximation error in the force calculation. In \eqref{eq:intro:force_ewald} we have seen that differentiation of the potential yields an extra term $\frac{1}{2}q_{\mi}(\frac{4\xi^2}{\eta})(\v x_{\mi}-\v x_{\ni})$. The term $(\v x_{\mi}-\bx_{\ni})$ can be bounded by $w$, the half width of the Gaussians. Hence, 
\begin{align}
|\F^\L_\SE-\F^\L|_{\rms}\lessapprox \frac{1}{2}\sqrt{Q}\left(\frac{4\xi^2w}{\eta}\right)A_\P e^{-\pi Pc^2/2},
\label{eq:approx_error_bound_force_temp}
\end{align}
where $A_\P\approx \sqrt{Q\xi L}/L$. On the other hand, using the definition of $\eta$ we have
\begin{align*}
\frac{4\xi^2w}{\eta} = \frac{4\xi^2w}{4\xi^2w^2/m^2}=\frac{m^2}{w}.
\end{align*}
To eliminate the truncation error, the term $e^{-(\pi k_{\infty}/\xi L)^2}$ in the error estimate \eqref{eq:truncation:kolafa_fourier} suggests that $M=2k_{\infty}\approx4\xi L$. Therefore, $w=Ph/2=PL/2M\approx P/8\xi$. Also, since $m^2\approx\pi P$, \eqref{eq:approx_error_bound_force_temp} can be simplified as
\begin{align}
|\F^\L_\SE-\F^\L|_{\rms}\lessapprox A_\F e^{-\pi Pc^2/2},\quad A_\F\approx 4\pi Q\sqrt{\frac{\xi^3}{L}}.
\label{eq:se:approx_error_bound_force}
\end{align}
Figure \ref{fig:se:approx_error_bound} (right) shows the experimental error bound and measured error in calculating the force as scaled with $A_\F$ using the SE method. The error is computed with the same 100 cases as in figure \ref{fig:se:approx_error_bound} (left) for the potential. For a specific error tolerance, the resulting $P$ value obtained from the error bound \eqref{eq:se:approx_error_bound_force} has to be rounded up to the next even integer to be able to utilize faster spreading/gathering routines available for even values of $P$. Hence, the fact that the approximation error bound is not as sharp as the bound for the potential, has no significant effect on the choice of $P$.
\begin{figure}[htbp]
  \tikzset{mark size=1}
  \begin{subfigure}[b]{0.49\textwidth}
    \centering 
    \includegraphics[width=\textwidth,height=.8\textwidth]{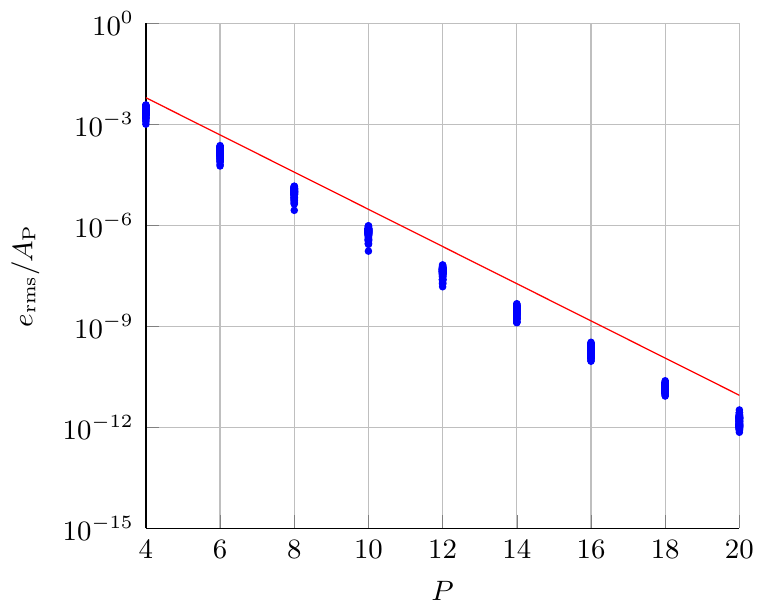}
  \end{subfigure}
\hfill
  \begin{subfigure}[b]{0.49\textwidth}
    \centering
    \includegraphics[width=\textwidth,height=.8\textwidth]{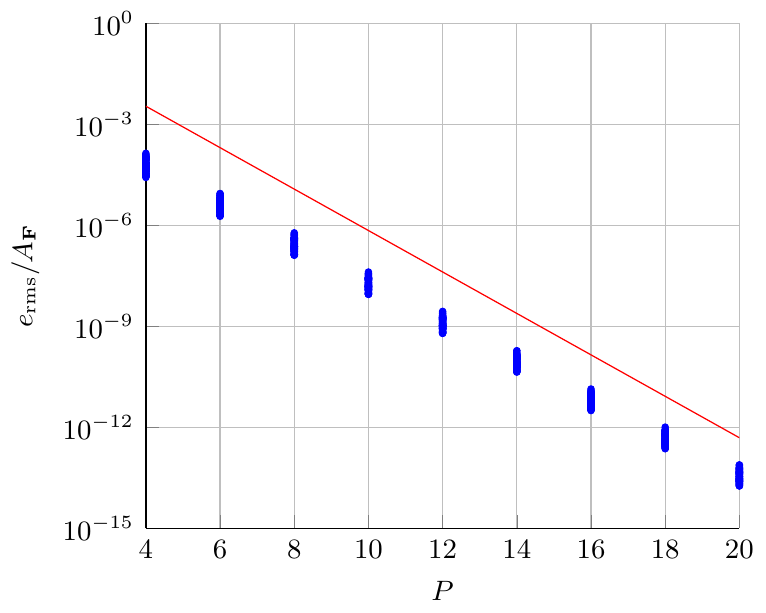}
  \end{subfigure}
  \caption {Scaled rms error in evaluating (left) the potential and (right) force. Scaled rms error confirms the dependency to $P$. The error is computed for each $P$ and for 100 different cases with $L\in\lbrace1,5,10,20,40\rbrace$, $N\in\lbrace200,400,800,1200\rbrace$, and $\xi\in\lbrace5,15,25,30,35\rbrace/L$. The scaled results satisfy the error bound $e^{-\pi Pc^2/2}$ (red line).}
  \label{fig:se:approx_error_bound}
\end{figure}

\begin{remark}
The error bounds \eqref{eq:se:approx_error_bound_potential} and \eqref{eq:se:approx_error_bound_force} suggest that the relative approximation error (obtained by scaling with $A_\P$ and $A_\mathrm{\bf F}$  respectively) has the form of $e^{-\pi Pc^2/2}$. Therefore with $P\approx10,20,24$, independent of the other parameters, the relative approximation error is less than $10^{-6}$, $10^{-12}$ and $10^{-15}$ respectively.
\end{remark}
\begin{remark}
The force defined throughout this article is computed based on \eqref{eq:se:force} but whenever we measure the absolute rms error in the SE and SPME methods in GROMACS, the force is scaled with the electrostatic conversion factor $1/4\pi\epsilon_0\approx138.935$.
\end{remark}

\subsection{Total error}
There are two types of errors involved in calculation of the electrostatic potential (energy or force) using the SE method, namely the truncation and approximation errors. The truncation errors are estimated by \eqref{eq:truncation:kolafa_real} and \eqref{eq:truncation:kolafa_fourier} and the approximation error is estimated by \eqref{eq:se:approx_error_bound_potential}, \eqref{eq:se:approx_error_bound_energy} and \eqref{eq:se:approx_error_bound_force}. 

To study the behavior of the truncation errors, we increase $P$ such that the approximation error is down to machine precision. We generate a uniformly distributed system of $N=800$ particles in a box of length $L=20$ and measure the absolute rms error as a function of the Ewald parameter for different grid sizes and cut-offs. Figure \ref{fig:se:total_error_potential} (left) shows how the total error in evaluating the potential for $\lbrace r_\c=4,M=32\rbrace$ and $\lbrace r_\c=5,M=64\rbrace$ (solid curves) follow the truncation error estimates \eqref{eq:truncation:kolafa_real} and \eqref{eq:truncation:kolafa_fourier} (dashed curves). Next, we investigate the effect of the approximation error in the total error. We choose a cloud-wall system (figure \ref{fig:se:systems}-middle) with $N=1200$ and $L=10$ and compute the total error for $\lbrace r_\c=5,M=64\rbrace$ and different $P$ values. Figure \ref{fig:se:total_error_potential} (right) shows how a smaller $P$ adds  approximation error to the total error. Again, the dashed curves represent the error levels obtained from theoretical estimates of truncation errors and the solid curves are the total error in evaluating the potential. We demonstrate in figure \ref{fig:se:total_error_force}, the total error in computing the force as a function of $\xi$ using the SE method. As we observe in both figure \ref{fig:se:total_error_potential} (left) and \ref{fig:se:total_error_force}, there is an excellent agreement between the actual and estimated errors.
\begin{figure}[htbp]
  \tikzset{mark size=1}
  \begin{subfigure}[b]{0.49\textwidth}
    \centering\includegraphics[width=\textwidth,height=.8\textwidth]{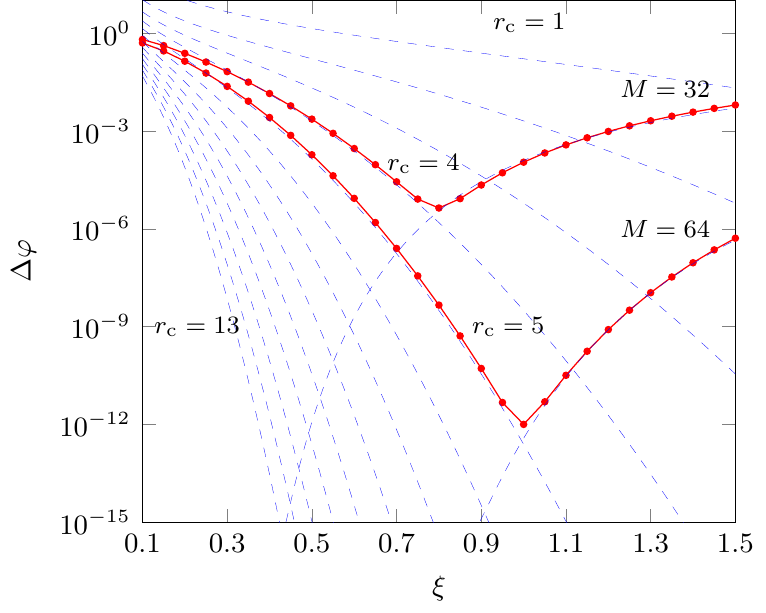}
  \end{subfigure}
  \hfill
  \begin{subfigure}[b]{0.49\textwidth}
    \centering\includegraphics[width=\textwidth,height=.8\textwidth]{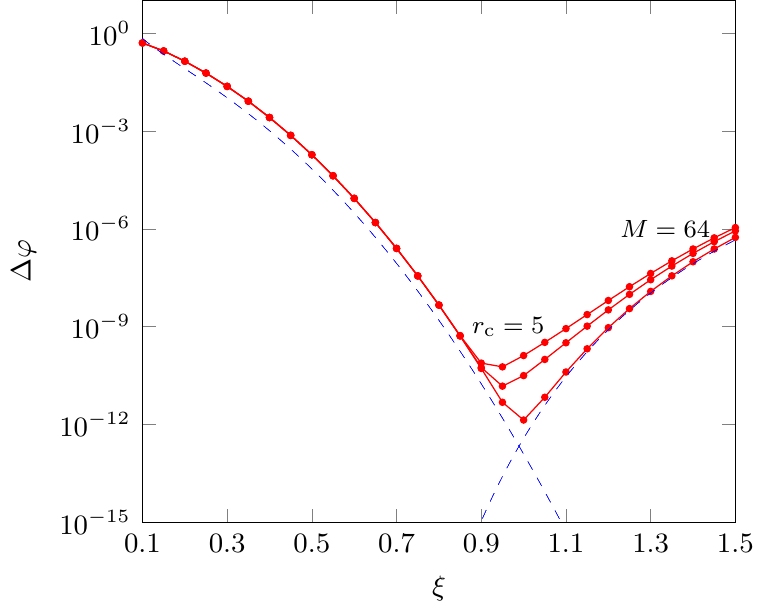}
  \end{subfigure}
  \caption{Measured and estimated absolute rms errors in computing the potential vs.\ $\xi$ using the SE method. The dashed curves show truncation error levels for different $r_\c$ and grid size $M$ based on the error estimates (\ref{eq:truncation:kolafa_real}a) and (\ref{eq:truncation:kolafa_fourier}a) and the solid curves represent the total measured error. A system of $N=800$ uniformly distributed particles in a box of length $L=20$ is used. (Left) $\lbrace r_\c=4,~M=32\rbrace$ and $\lbrace r_\c=5,~M=64\rbrace$. The parameter $P$ is chosen such that the approximation error is negligible. (Right) $r_\c=5$ and $M=64$. From top to bottom $P=18,20,24$.}
  \label{fig:se:total_error_potential}
\end{figure}
\begin{figure}[htbp]
  \tikzset{mark size=1}
  \centering
\includegraphics[width=0.45\textwidth,height=.4\textwidth]{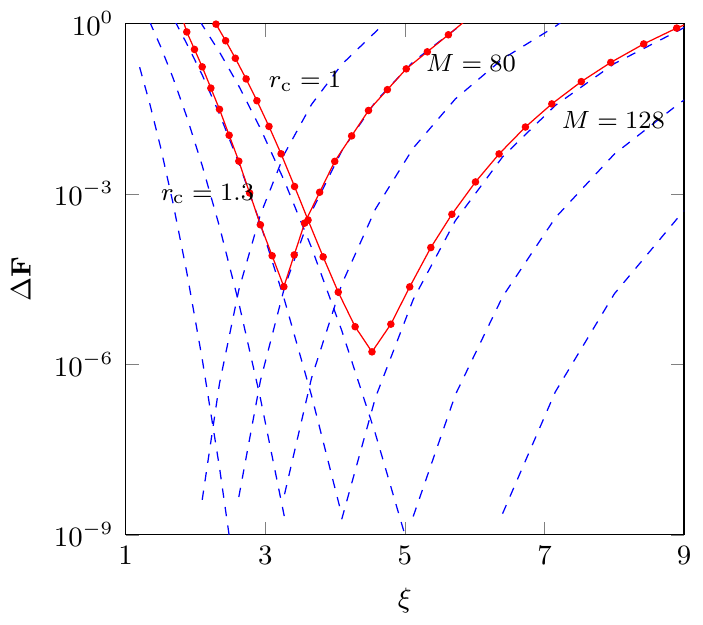}
\caption{Measured and estimated absolute rms errors in computing the force vs.\ $\xi$ using the SE method. The dashed curves show truncation error levels for different $r_\c$ and grid size $M$ based on the error estimates (\ref{eq:truncation:kolafa_real}c) and (\ref{eq:truncation:kolafa_fourier}c) and the solid curves represent the total measured error with $\lbrace r_\c=1,~M=128\rbrace$ and $\lbrace r_\c=1.3,~M=80\rbrace$. The parameter $P$ is chosen such that the approximation error is negligible. A cloud-wall system with $N=1200$ particles in a box of length $L=10$ is used (figure \ref{fig:se:systems}-middle).}
\label{fig:se:total_error_force}
\end{figure}

\subsection{Momentum conservation}
The conservation of momentum and energy are important criteria in the simulation of particles. Narrowing many events which affect the conservation to the methods that we use, some methods conserve momentum while others conserve energy. The conservation of energy seems more demanding than the conservation of momentum in molecular dynamics simulations. In classical Newtonian dynamics, the energy of an isolated system is conserved. Therefore, the molecular dynamics as a numerical solution of classical dynamics, has to conserve energy as well.

In Darden \etal version of the PME method \cite{pme}, the forces are interpolated symmetrically in space and the same function is used to distribute charges and to interpolate forces. Therefore, the PME method conserves momentum and not energy. On the contrary, the SPME method conserves energy and not momentum. Essmann \etal showed \cite{spme}, as an example, that the sum of electrostatic forces is not zero and that the simulation leads to a small random particle drift. This, however, can be removed by subtracting the drift from all particle before doing another time iteration \cite{spme}. The momentum conservation can be examined by computing the sum of forces in a charge neutral system. Due to the fact that in the SE method the force is obtained through analytic differentiation of the energy, the method conserves energy, however, it does not conserve momentum. To have an understanding of the magnitude of the sum of forces, we compute the electrostatic force for a pair of oppositely charged particles. The experiment shows that the net force is not zero but is in the same order as the rms error. This is in agreement with the observation in \cite{spme}.

\section{Errors in the SPME method}
\label{section:error_spme}
In this section, we briefly review the SPME method and give all the
formulas needed for its implementation, before we summarize the
formulas for the SE and SPME method side by side in Table \ref{tab:compare}. We
discuss the character of the approximation errors in the SPME method
in comparison to those of the SE method in order to understand the different
requirements on the FFT grid size in the two methods.

\subsection{Methodology}
\label{subsection:methodology}
The SPME method \citep{spme} is a reformulation of the PME method \citep{pme} which
itself is inspired by the P$^3$M method \citep{p3m}, and follows the structure of
the algorithm described in section \ref{section:fftbased}. The specific choice of window
functions is cardinal B-splines of different orders, as will be
described below. The electrostatic force is obtained by analytic
differentiation of the energy as generically described in section \ref{subsec:force_calc} 
and detailed for the SE method in section \ref{section:se:force}.

Suppose $N$ particles with charges $q_{\ni}$ at positions $\v x_{\ni}$, $\ni=1,2,\ldots,N$, are interacting with each other in a unit box and let $\v x=(x,y,z)=(a_1h,a_2h,a_3h)$, $a_i\in\lbrace0,1,2,\ldots\rbrace$ for $i=1,2,3$, with $h=1/M$ for a positive integer $M$. The cardinal B-spline of order $p=2$ is defined as 
\begin{align*}
\mathcal{M}_2(u)=
\left\{
\begin{array}{ll}
  1-|u-1|, & 0\leq u\leq 2, \\
  0, & \text{otherwise},
\end{array}
\right.
\end{align*}
and for $p>2$ is defined recursively as
\begin{align*}
\mathcal{M}_p(u) = \frac{u}{p-1}\mathcal{M}_{p-1}(u)+\frac{p-u}{p-1}\mathcal{M}_{p-1}(u-1).
\end{align*}
The charge distribution function $\sum_{\ni}q_{\ni}\delta(\v x-\v x_{\ni})$ is smeared by a convolution with an order $p$ cardinal B-spline on a uniform grid of size $M^3$ as
\begin{align}
Q(\v x)=\sum_{\ni=1}^Nq_{\ni}\bm{\mathcal{M}}_p(\v x_{\ni}-\v x)_{\ast} = \sum_{\ni=1}^Nq_{\ni}\mathcal{M}_p(x_{\ni}-x)_{\ast}\mathcal{M}_p(y_{\ni}-y)_{\ast}\mathcal{M}_p(z_{\ni}-z)_{\ast}.
\label{eq:spme:grid}
\end{align}
Here, as before, $^{\ast}$ denotes that periodicity is applied in all direction. The gridded distribution function $Q(\v x)$ acts in the same way as $H(\v x)$, cf.\ \eqref{eq:se:grid}, introduced in the SE method. The discrete Fourier transform of $Q$, denoted by $\wh{Q}$, can then be computed via a 3D FFT. The scaled gridded distribution function now is given by a multiplication in Fourier space,
\begin{align}
\wh{\wt{Q}}(\v k) = B(\v k)\frac{e^{-k^2/4\xi^2}}{k^2}\wh{Q}(\v k),
\label{eq:spme:scale}
\end{align}
where
\begin{align*}
B(k_1,k_2,k_3)=|b_1(k_1)|^2|b_2(k_2)|^2|b_2(k_3)|^2,\quad b_j(k_j)=\frac{e^{(p-1)k_jh}}{\sum_{\ell=0}^{p-2}\mathcal{M}_p((\ell+1)h)e^{\ii k_j\ell h}}.
\end{align*}
An inverse 3D FFT transforms $\wh{\wt{Q}}$ to the real space (denoted by $\wt{Q}$) and the result is then interpolated using another cardinal B-spline to recover the potential at target points, i.e.
\begin{align}
\varphi^{\L}_\SPME(\v x_\mi) = \dfrac{4\pi}{L^3}\sum_\ni \wt{Q}(\v x_\ni)\bm{\mathcal{M}}_p(\v x_{\ni}-\v x_\mi)_{\ast}.
\label{eq:spme:potential}
\end{align}
As mentioned previously, the force field is derived with analytic differentiation when needed \cite{spme}. In table \ref{tab:compare}, we give the formulas for the SE and SPME methods, where the steps correspond to those introduced in Algorithm \ref{alg:fftbased}. 

\begin{table}[h!]
\begin{center}
\begin{tabular}{|c|c|c|c|c|}
\hline
\textbf{steps} & \textbf{SE} & \textbf{Eq} & \textbf{SPME} & \textbf{Eq}. \\ 
\hline 
1 & regular grid of size $M^3$ &  & regular grid of size $M^3$ & \\ 
2 & $\left(\frac{2\xi^2}{\pi\eta}\right)^{3/2}e^{-2\xi^2|\v x_{\ni}-\v x|_{\ast}^2/\eta}$ & \ref{eq:se:grid} & $\bm{\mathcal{M}}_p(\v x_{\ni}-\v x)_{\ast}$ & \ref{eq:spme:grid}\\
3 & Apply 3D FFT & & Apply 3D FFT & \\
4 & $k^{-2}e^{-(1-\eta)k^2/4\xi^2}$ & \ref{eq:se:scaling} & $B(\v k)k^{-2}e^{-k^2/4\xi^2}$ &  \ref{eq:spme:scale}\\
5 & Apply 3D IFFT & & Apply 3D IFFT & \\
6 & $\left(\frac{2\xi^2}{\pi\eta}\right)^{3/2}e^{-2\xi^2|\v x_{\ni}-\v x|_{\ast}^2/\eta}$ & \ref{eq:se:potential_discrete} &
$\bm{\mathcal{M}}_p(\v x_{\ni}-\v x)_{\ast}$ & \ref{eq:spme:potential}\\
\hline
\end{tabular} 
\end{center}
\caption{A stepwise comparison of the SE and SPME methods with the steps as
given in Algorithm 1 and corresponding equation numbers. The number of
points in the support of the truncated Gaussian ($P$, which sets $\eta$ from equation \eqref{eq:eta}) must be chosen for the SE
method and the degree of the B-splines ($p$) for the SPME method.}
\label{tab:compare}
\end{table}

\subsection{Approximation error}
\label{section:error_spme:approximation_error}
In the SPME method, as a result of using $p$th order B-splines in the spreading/gathering steps, an approximation error of order $h^p$, $h=L/M$ is added. Hence, to reduce this type of error, one needs to increase the grid size. Consequently, the cost of scaling and specially FFT steps increases. Essmann \etal \cite{spme} used B-splines of order $p=3$ for low accuracies ($<10^{-4}$), $p=5$ for moderate accuracies ($<10^{-6}$) and $p=7$ for high accuracies ($<10^{-8}$) but they did not present a systematic way of choosing B-spline orders since they had no estimate of the approximation error available. In 2012, Ballenegger \etal \cite{ballenegger} showed that the P$^3$M method can be converted into the SPME method with a modification in the influence function. Therefore, the error estimates of the P$^3$M method can be used for the SPME method as well. Later, Wang \etal \cite{wang} presented an error estimate for the approximation error, but their estimate was not simple enough to be used in practice.

The approximation error in the SE method is independent of the grid size and is only a function of the number of points in the support of suitably scaled and truncated Gaussians. Hence, if the grid is coarser, the Gaussians are made wider. As a result of this difference, a much smaller FFT grid size is used in the SE method compared to the SPME method. This is an important advantage of the SE method over the SPME method specially on parallel architectures. Parallel FFT algorithms require all-to-all communications between processes and the number of communications increases as the square number of processes. Therefore, there have been several attempts to replace FFTs by iterative solvers or to employ a Fast multipole method (FMM) \cite{gromacs_parallel}. We shall return to this discussion in section \ref{section:parallel}. On the other hand, the spreading/gathering steps are more expensive in the SE method compared to the SPME method, even though their costs are substantially reduced by employing the Fast Gaussian Gridding. In section \ref{section:se_spme_time_comparison} we will see how changes in parameters will influence the behavior of both methods.

To visualize the dependency of the SE and SPME methods to grid size, in figure \ref{fig:se:se_spme_gridsize} we plot the relative rms error in evaluation of the force as a function of total number of grid points in three directions, $M^3$, for both methods and with two different Ewald parameters. The figure shows the spectral accuracy in the computation of the force in the SE method and polynomial accuracy in the SPME method. Note that if a smaller value of $P$ were chosen for the SE method, the
error curve would follow the current curve down to a certain accuracy
level, and then levels out.

\begin{figure}[htbp]
  \tikzset{mark size=1.5}
  \begin{subfigure}[b]{0.49\textwidth}
    \centering
    \includegraphics[width=\textwidth,height=.8\textwidth]{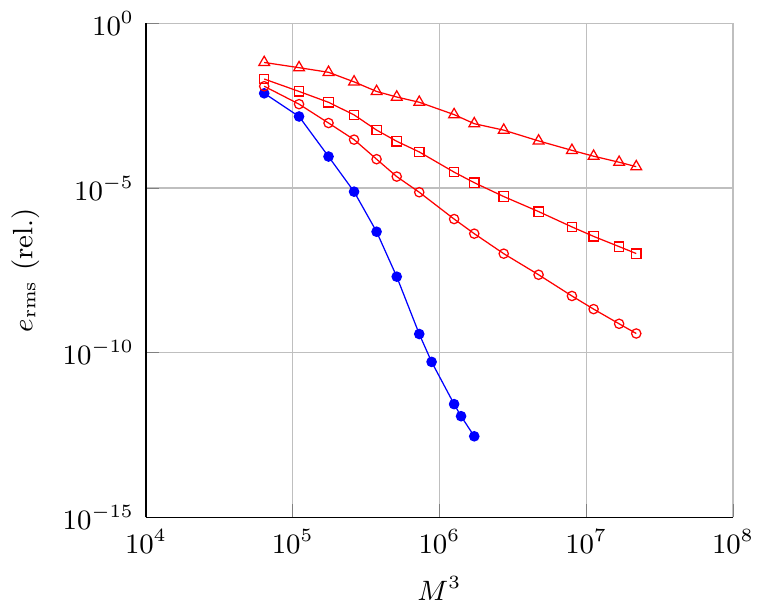}
  \end{subfigure}
  \hfill
  \begin{subfigure}[b]{0.48\textwidth}
    \centering
    \includegraphics[width=\textwidth,height=.8\textwidth]{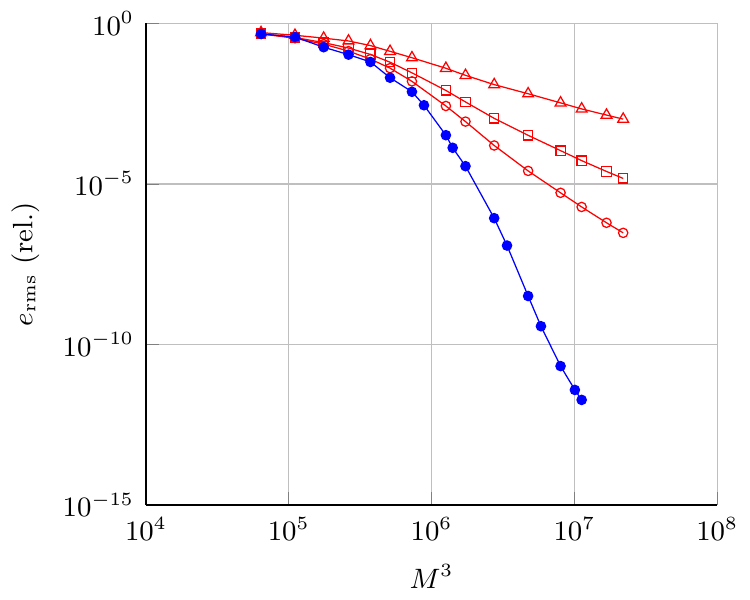}
  \end{subfigure}
  \caption{Relative rms error in evaluation of the force as a function of $M^3$ for the SE and SPME methods with (left) $\xi=3$ and (right) $\xi=6$. A uniform system of $N=81\,000$ particles in a box of length $L=9.3222$ is used. The red curves represent the SPME method with B-splines of order $p=3,5,7$ from top to bottom and the blue curve represents the SE method with $P=24$ points in the support of the Gaussians in each direction.}
\label{fig:se:se_spme_gridsize}
\end{figure}

\section{Runtime estimates}
\label{section:runtime_estimate}
So far, we have explained the fundamentals of the SE method and its error. In this section, we estimate the runtime of different parts of the
electrostatic calculations to later aid us in the choice of the
optimal Ewald parameter for our implementation of the SE method. For this, we implement the algorithm in GROMACS ver.\ 5.1 \cite{gromacs}. GROMACS is one of the fastest Molecular Dynamics simulation packages available which serves many users. It is efficient, highly optimized for serial use as well as parallel using MPI and/or multi-threading. Three different systems are examined; A uniformly distributed system with $N\in\lbrace1200, 3000, 81\,000, 1\,029\,000\rbrace$ particles, a cloud-wall system with $N\in\lbrace1200, 1\,012\,500\rbrace$ particles used in \cite{axel1} and an isolated dense-cloud system with $N=1200$ particles (see figure \ref{fig:se:systems}). The cloud-wall and isolated systems are generated artificially to allow for strong long range interactions.
\begin{figure}[htbp]
\centering
  \begin{subfigure}[b]{0.3\textwidth}
    \centering 
    \includegraphics[width=1\linewidth]{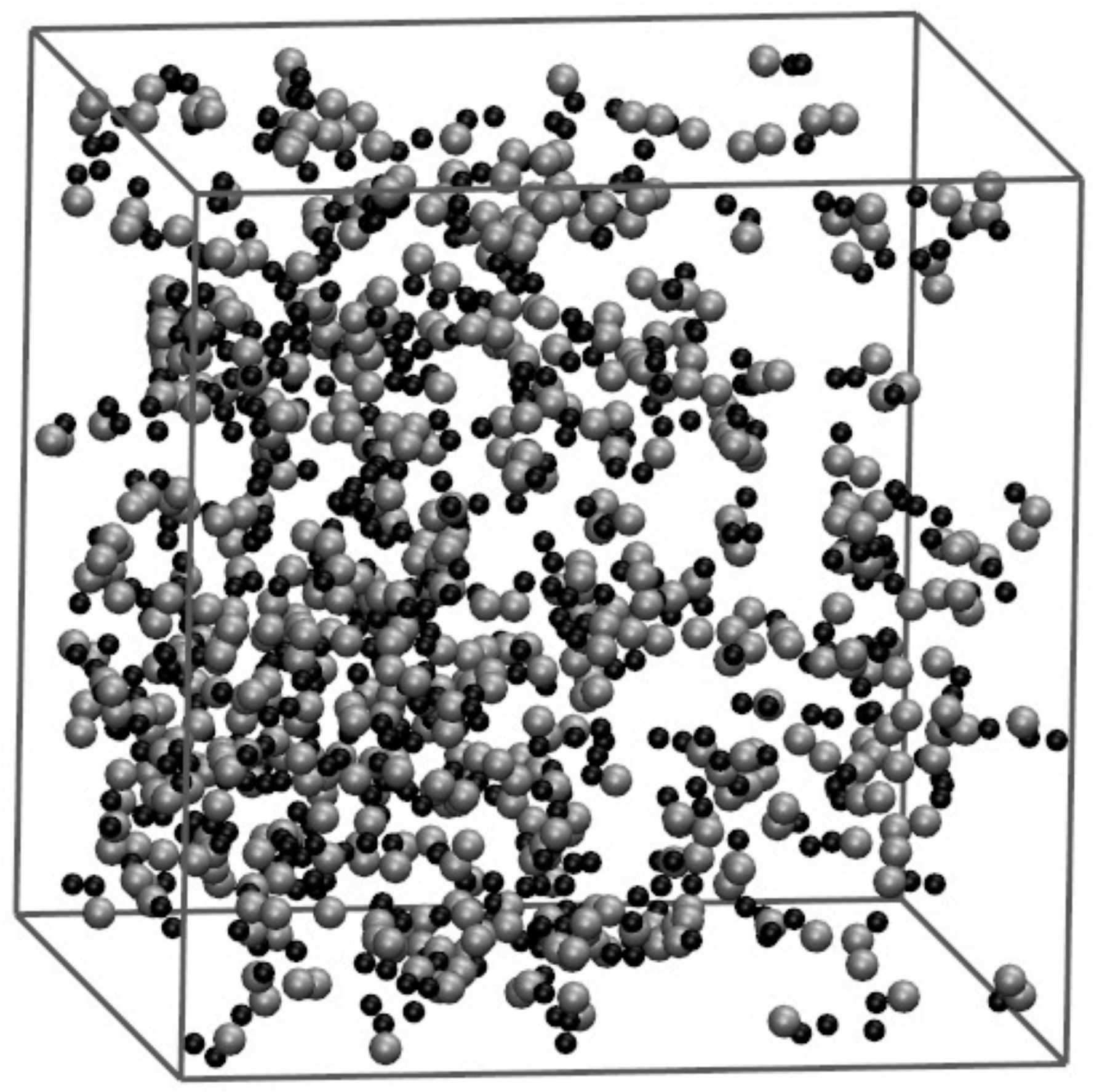}
  \end{subfigure}
  \begin{subfigure}[b]{0.3\textwidth}
    \centering 
    \includegraphics[width=1\linewidth]{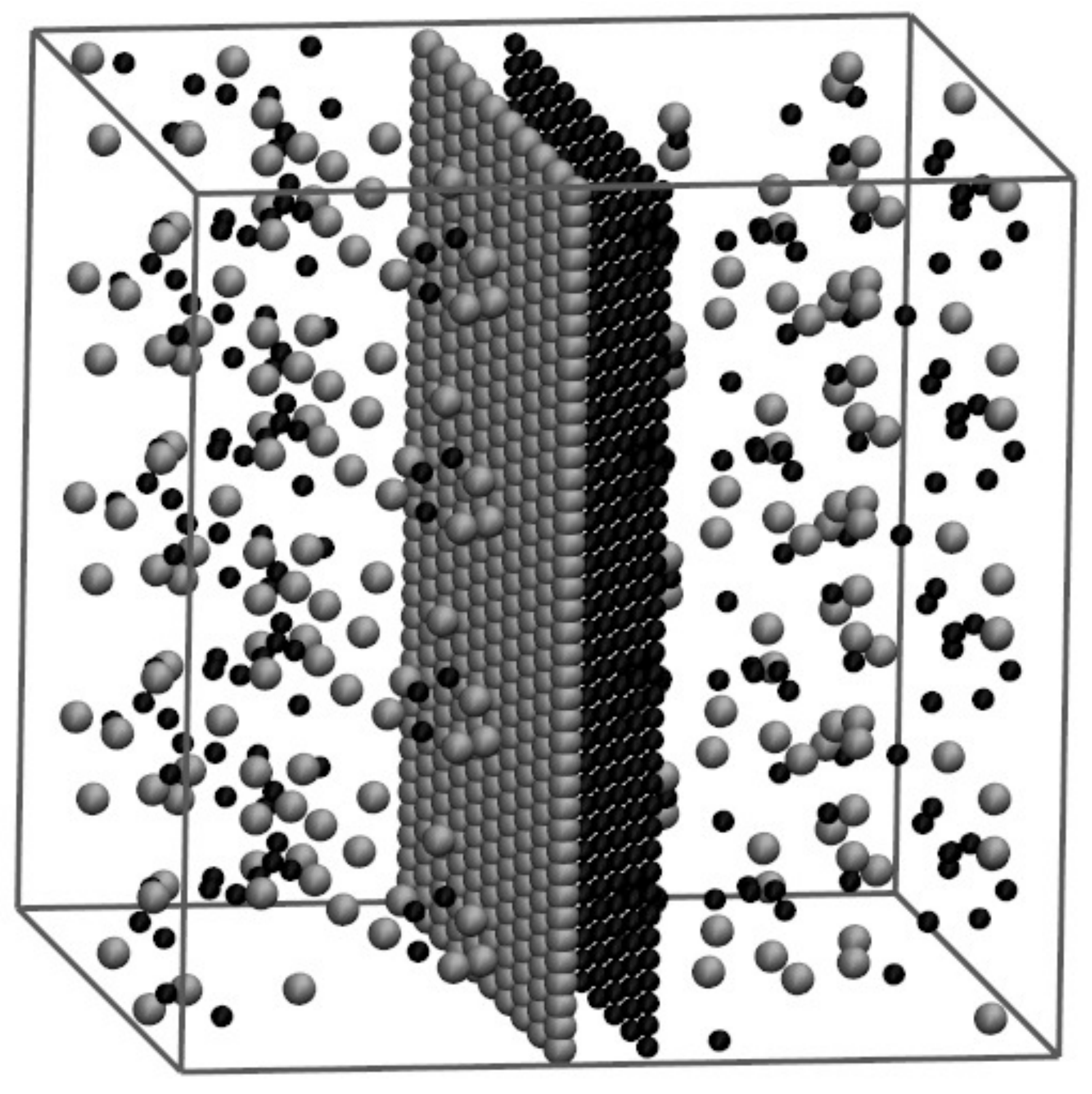}
  \end{subfigure}
  \begin{subfigure}[b]{0.3\textwidth}
    \centering 
    \includegraphics[width=1\linewidth]{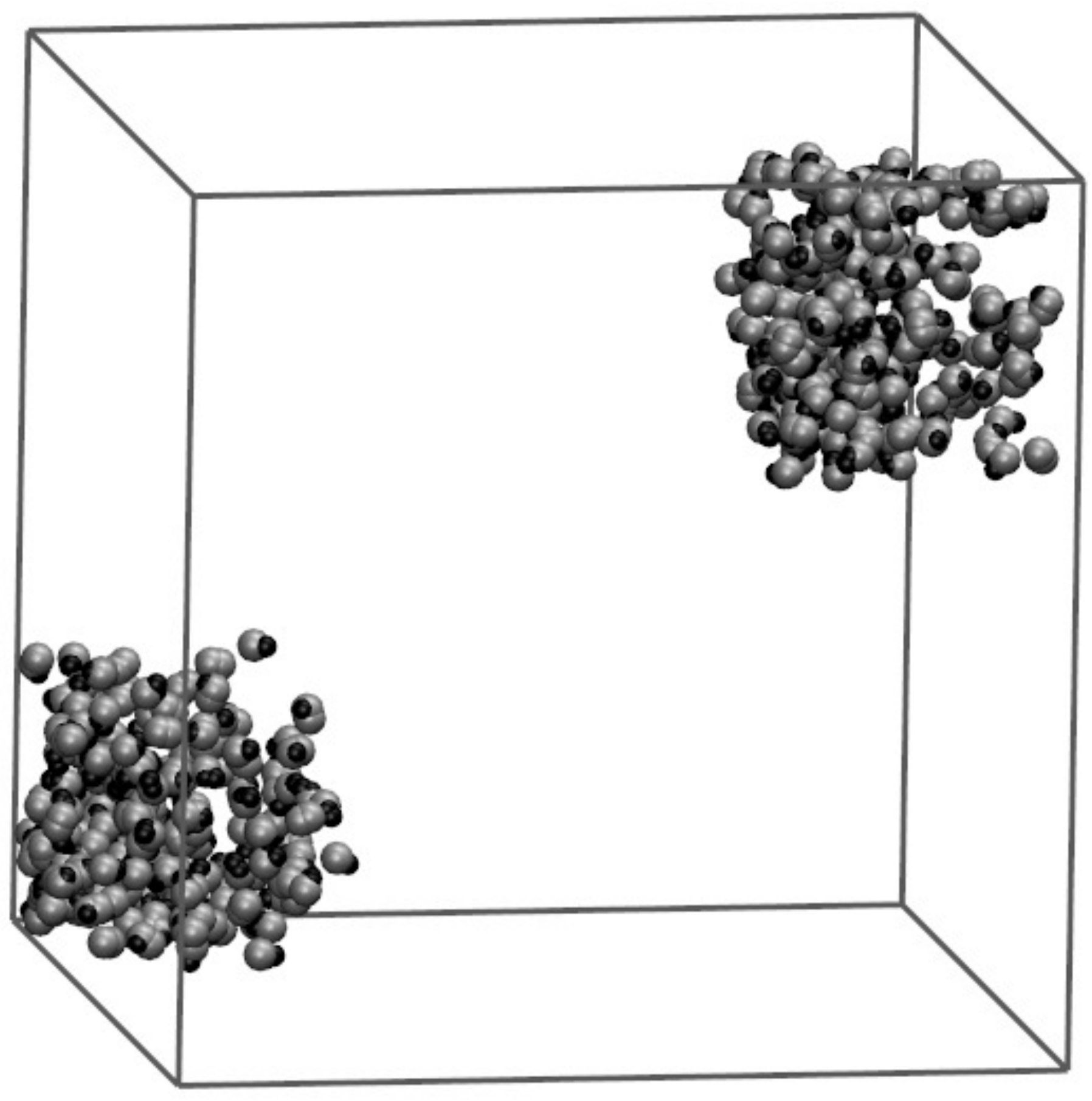}
  \end{subfigure}
  \caption{Three systems with $N=1200$ particles. (Left) A homogeneous system of particles. (Middle) a cloud-wall system of particle (borrowed from \cite{axel1}). (Right) an isolated dense-cloud system of particles.}
  \label{fig:se:systems}
\end{figure}

Finding the optimal Ewald parameter is always a concern in Ewald methods. The optimal parameter balances the runtime of the real and Fourier space parts of the sum. This optimal value depends not only on the algorithms that are used but also on the implementation of the methods. Therefore, what we present here might require a modification with a different implementation of the SE method and also with the use of other MD packages. Here, we strive to estimate the runtime of different parts of the Electrostatic
calculations. The results obtained here are useful to estimate the Ewald parameter while using the SE method. We start by generating a charge neutral system of $N=3000$ uniformly distributed particles placed in a box of length $L=3.1074$. To generate other uniformly distributed systems of particles, this system is scaled up  such that the charge density remains constant, i.e.\ $N/L^3\approx3000/30=100$. 

All runs (except those related to figure \ref{fig:beskow}) are performed on a desktop machine with 8 GB of memory and an Intel Core i7-3770 3.40 GHz CPU with four cores and eight threads. GROMACS was built with the Gnu C Compiler version 4.8.3. 

\subsection{Timing the real space sum}
Here we denote by $\n$, the average number of nearest neighbors within a cut-off radius $r_\c$ around each particle. As the points are uniformly distributed, the maximum and minimum number of nearest neighbors are approximately equal and therefore we have
$$\n=\frac{\frac{4}{3}\pi r_\c^3}{L^3}N.$$
The computation of the real space sum in GROMACS consists of two parts; neighbor searching and force calculation. The neighbor searching is the process of finding the neighboring particles which lie in a specified cut-off radius for each particle. In GROMACS, this step has a complexity of $\order{N}$ (though with a large prefactor of 125), cf.\ \cite{gromacs_manual}. The experiment (performed in GROMACS) suggests the following for the runtime in the neighbor searching step,
$$t_{\text{ns}}= c_{\text{ns}}\n N,$$
where $c_{\text{ns}}\approx 1.4\cdot10^{-8}$.
\begin{figure}[htbp]
  \tikzset{mark size=1.5}
  \begin{subfigure}[b]{0.49\textwidth}
    \centering 
    \includegraphics[width=1\linewidth]{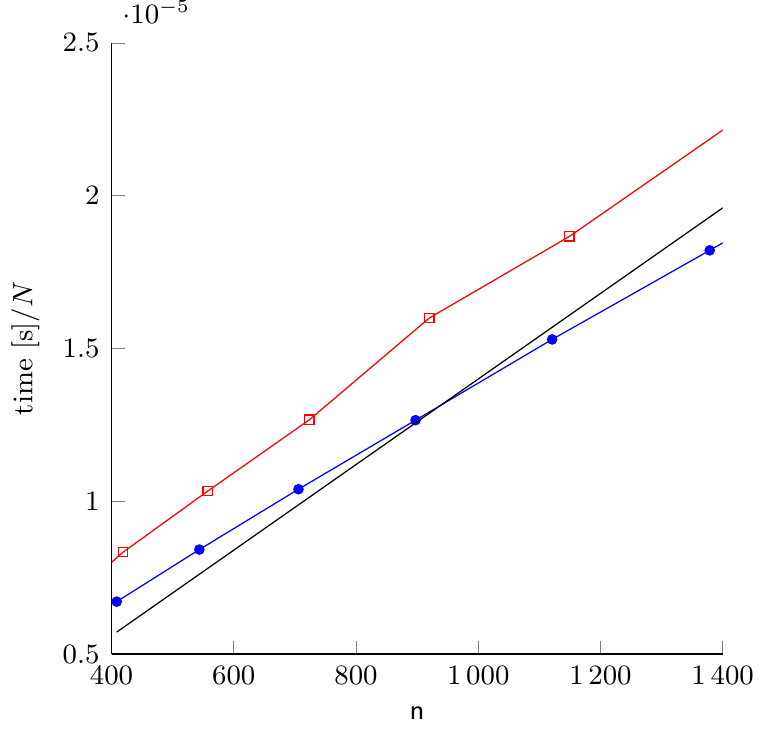}
  \end{subfigure}
  \hfill
  \begin{subfigure}[b]{0.49\textwidth}
    \centering 
    \includegraphics[width=.95\linewidth]{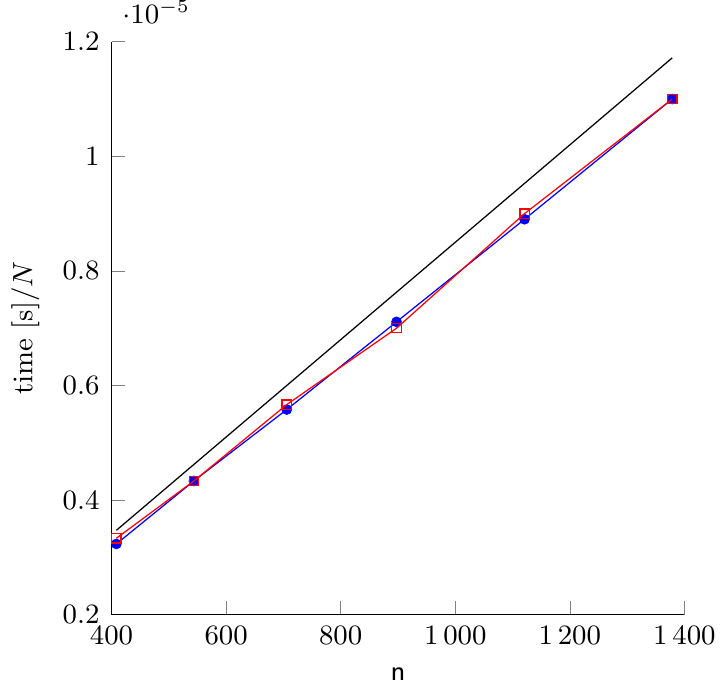}
  \end{subfigure}
  \caption{The experiment suggests ${\cal O}(N)= c\n N$ for (left) neighbor searching and (right) real space force calculation. The black line represents the estimated time. The red and blue dotted lines are the measured runtimes scaled with the number of particles $N=3000$ and $N=81\,000$ respectively.}
  \label{fig:ns_force_time}
\end{figure}

The calculation of the so called short-range or real space part of the Ewald sum has ${\cal O}(N)$ complexity as well, and the experiment suggests that the runtime follows
\begin{align*}
t_{\text{force}}= c_{\text{force}} \n N,
\end{align*}
where $c_{\text{force}}\approx 10^{-8}$. Figure {\ref{fig:ns_force_time} (left) shows the scaled runtime of the neighbor searching and figure {\ref{fig:ns_force_time} (right) shows the dependency of the scaled runtime of the real space sum calculation as a function of number of nearest neighbors. The computational time of the real space sum can then be estimated as
\begin{equation}
t_{\text{real}} = t_{\text{ns}}+t_{\text{force}}=c_{\text{real}}\n N,
\end{equation}
with $c_{\text{real}}\approx2.4\cdot10^{-8}$. Again, the constants $c_{\text{ns}}$, $c_{\text{force}}$ and $c_{\text{real}}$ are found by numerical experiments and depend on the implementation and the machine that is used.

We should also mention that the runtime of the real space sum presented here includes the runtime of other computations which do not actually belong to the real space sum but they cannot be excluded in GROMACS. Yet GROMACS internal parameters are chosen in a way that these calculations are small and fixed for all different runs.
\subsection{Timing the Fourier space sum}
\label{section:runtime_estimate_fourier}
The Fourier space computation consists of three different routines; FFT/IFFT, spreading/gathering and solving (scaling) routines. The computational costs of the FFT and IFFT routines are similar and are of the order $M^3\log M^3$. Since GROMACS performs a real to complex FFT and a complex to real IFFT, we can write 
\begin{align*}
t_{\text{fft-ifft}}=c_{\text{fft-ifft}}\frac{M^3}{2}\log(\frac{M^3}{2}),
\end{align*}
with $c_{\text{fft-ifft}}\approx4.5\cdot10^{-9}$.

Another step in the evaluation of the Fourier space sum using the SE method is the scaling step which has ${\cal O}(M^3)$ computational complexity. The runtime estimate in this step is in the form of
\begin{align*}
t_{\text{solve}}=c_{\text{solve}}M^3,
\end{align*}
and $c_{\text{solve}}\approx2\cdot10^{-9}$.

In the spreading (and gathering) step, the number of operations required to spread to (or interpolate from) the $P^3$ neighboring points of a source point are equal and scale as ${\cal O}(NP^3)$. For some $P$ values, it is possible to accelerate the routines with the aid of single instruction, multiple data (SIMD) instructions. Employing SIMD intrinsics will decrease the computational time for those $P$ values. Considering this timing complicates our analysis but for simplicity we suggest the following form
\begin{align*}
t_{\text{sp-ga}}=c_{\text{sp-ga}}NP^3,
\end{align*}
and $c_{\text{sp-ga}}\approx5\cdot10^{-9}$. Note that in the timing of the spreading and gathering steps, the runtime of the pre-computation step is also included. The pre-computation step includes the computation of the exponentials in the FGG routine. The total time spent in the Fourier space part of the sum can then be estimated as 
\begin{align}
 t_{\text{fourier}} &= t_{\text{fft-ifft}}+t_{\text{sp-ga}}+t_{\text{solve}}\nonumber \\
&=c_{\text{fft-ifft}}\frac{M^3}{2}\log(\frac{M^3}{2})+c_{\text{sp-ga}}NP^3+c_{\text{solve}}M^3,
\end{align}
where $c_{\text{fft-ifft}}\approx4.5\cdot10^{-9}$, $c_{\text{sp-ga}}\approx2\cdot10^{-9}$, and $c_{\text{solve}}\approx5\cdot10^{-9}$.

To assess reliability of the runtime estimates for the real and Fourier space sum presented above, a system of $N=81\,000$ uniformly distributed charged particles in a box of length $L=9.3222$ is considered. We measure the runtime in evaluating the force for different Ewald parameters and to achieve a relative rms error of $\approx10^{-5}$. Figure \ref{fig:se:total_time_81000_1e5} (left) shows that the estimates can predict actual runtimes rather well. It also suggests an optimal value of $\xi\approx6.75$ (corresponding to $r_\c=0.60$). 

We repeat the same experiment but now to achieve a relative rms error of $\approx10^{-12}$. Figure \ref{fig:se:total_time_81000_1e5} (right) suggests another  optimal value, $\xi\approx6.31$ (corresponding to $r_\c=0.9$). Comparing the plots in figure \ref{fig:se:total_time_81000_1e5} we observe that for the SE method, the Ewald parameter has a weak dependency to the error tolerance. This is again in contrast with the SPME method, see figure \ref{fig:se_spme_time_comparison:time_xi_se_spme_1e5} where the dependence is much stronger.
\begin{figure}[htbp]
  \tikzset{mark size=1.5}
  \begin{subfigure}[b]{0.49\textwidth}
    \centering \includegraphics[width=\textwidth,height=.8\textwidth]{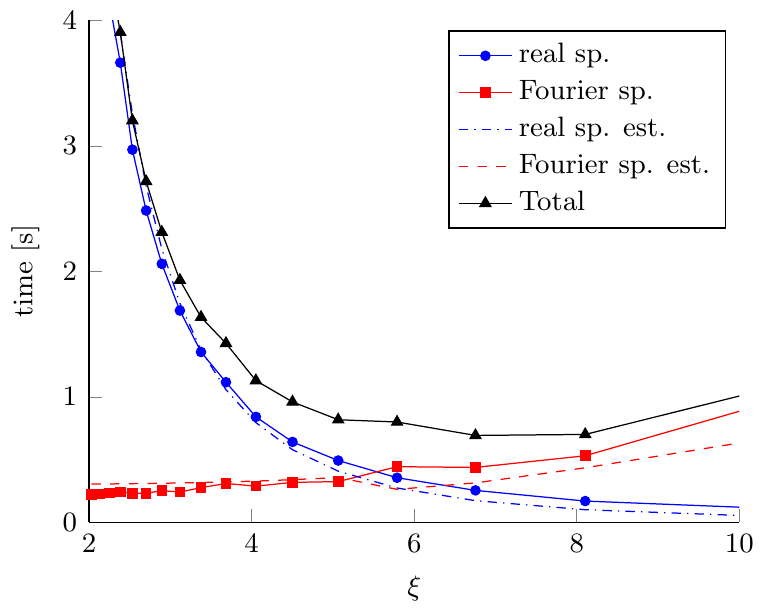}
  \end{subfigure}
  \hfill
  \begin{subfigure}[b]{0.49\textwidth}
    \centering 
    \includegraphics[width=\textwidth,height=.8\textwidth]{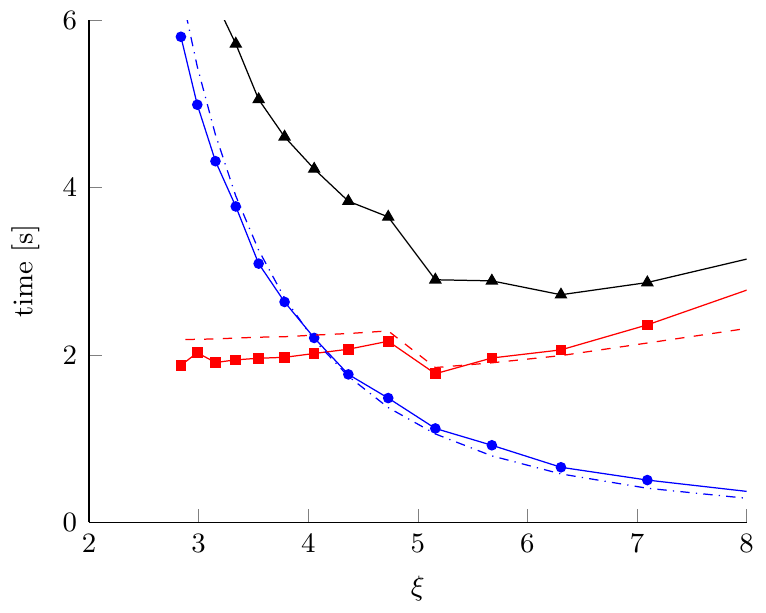}
  \end{subfigure}
  \caption{Real space, Fourier space, and total runtime as a function $\xi$ for a system of $N=81\,000$ uniformly distributed particles in a box of length $L=9.3222$. (Left) To achieve a relative rms error of $\approx10^{-5}$. $M=50,\ldots,250$ and $P\in\lbrace8,10\rbrace$. The optimal Ewald parameter is $\xi\approx6.75$ which corresponds to $r_\c=0.60$. (Right) To achieve a relative rms error of $\approx10^{-12}$, $M=100,\ldots,280$ and $P\in\lbrace20,22\rbrace$. The optimal Ewald parameter is $\xi\approx6.31$ which corresponds to $r_\c=0.90$.}
  \label{fig:se:total_time_81000_1e5}
\end{figure}

To convey the effect of system setting on the optimal Ewald parameter we find the optimal $\xi$ for three different systems in figure \ref{fig:se:systems} with the same number of particles and box size. We use systems with $N=1200$ oppositely charged particles and boxes of length $L=10$. Figure \ref{fig:se:three_systems_1e8} shows that, as expected, the optimal $\xi$ is only slightly different for each system. The density of the clouds in the isolated-cloud system is much higher than the density of other systems. Therefore, to keep the number of nearest neighbors fixed, the radius cut-off $r_\c$ is smaller for the isolated-cloud system and as a result, $\xi$ is larger. On the other hand, the cloud-wall system is designed in the way that the number of particles located on the walls are 4 times more than the number of particles in each cloud. Due to the 2D structure of the walls, $r_\c$ has to be increased to
have the same average number of near neighbors. This gives rise to a smaller Ewald parameter. Figure \ref{fig:se:three_systems_1e8} shows the optimal Ewald parameter obtained for three different cases.
\begin{figure}[htbp]
  \tikzset{mark size=1}
  \begin{subfigure}[b]{0.32\textwidth}
    \centering
    \includegraphics[width=\textwidth,height=.8\textwidth]{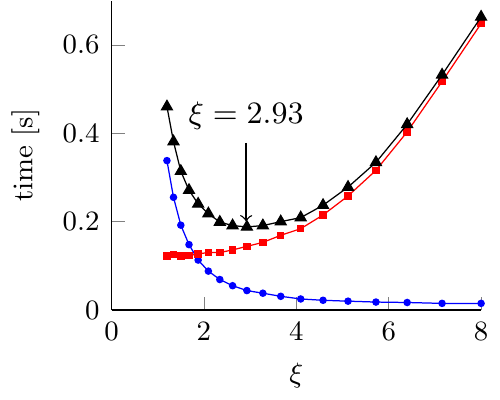}
  \end{subfigure}
  \begin{subfigure}[b]{0.32\textwidth}
    \centering
    \includegraphics[width=\textwidth,height=.8\textwidth]{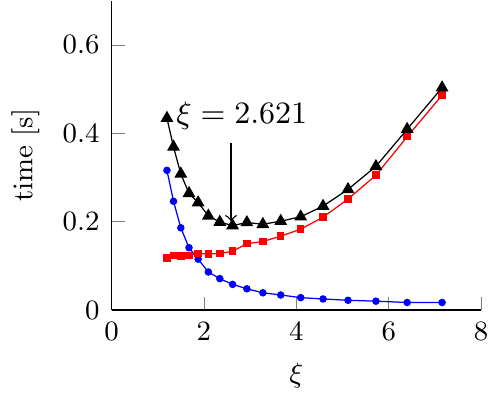}
  \end{subfigure}
  \begin{subfigure}[b]{0.32\textwidth}
    \centering
    \includegraphics[width=\textwidth,height=.8\textwidth]{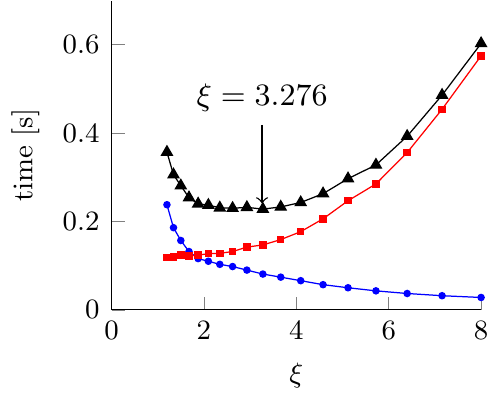}
\end{subfigure}
\caption{Real space ($\color{blue}\bullet$), Fourier space ($\color{red}\blacksquare$) and total runtime ($\color{black}\blacktriangle$) as a function of $\xi$ (left) for the uniform system, (middle) cloud-wall system and (right) two isolated-cloud system of particles to achieve a relative rms error of $\approx10^{-8}$. $N=1200$, $L=10$, $P=18$, $M=40,\ldots,220$.}
\label{fig:se:three_systems_1e8}
\end{figure}

\section{Comparison of the SE and SPME methods - serial implementation}
\label{section:serial}
There are many different FFT-based methods available to calculate the long-range electrostatics in periodic systems of particles. 
However, one of the most efficient methods available is still the SPME method. Moreover, one of the fastest and highly optimized implementations of the SPME method is available in GROMACS. Considering the man-hour time spend in optimizing the method, this is a difficult challenge for the SE method. In this section, we consider the serial implementation of the SE and SPME methods. Also, due to the fact that in our comparisons, the real space part of the sum is the same in the computation of the electrostatic interactions, from now on, we ignore the runtime of the real space part and only present results of the comparison in the Fourier space part of the Ewald sum.

In sections \ref{section:truncation} and \ref{section:error_se:approximation_error} the estimates are presented for the absolute errors though, in this and the following sections we measure relative rms errors. To obtain estimates for relative errors one can, however, use the magnitude of the potential and force provided in section \ref{section:error_se:approximation_error}.

\subsection{SE-SPME stepwise comparison}
We use a uniform system of $N=81000$ particles located in a box of length $L=9.3222$. This system is an equilibrated system of $1000$ water molecules in a box of length $31.074${ \AA} $=3.1074~nm$ with charges $-0.8476$ and $0.4238$ for Oxygen and Hydrogen atoms respectively. Figure \ref{fig:serial:se_spme3} shows the runtime in evaluating the force as a function of relative rms error with $P=16$ in the SE method and $p=3,5$ in the SPME method. The figure includes only the runtime of the FFT and \textit{solve} steps and not the spreading/gathering step. Also we present only the error committed in the Fourier space part of the sum.

\subsubsection*{FFT/IFFT} 
Since both methods are implemented in GROMACS, the same FFT routine is used for both methods. Figure \ref{fig:serial:se_spme3} (left) represents the total time spent in the FFT and IFFT routines. Since the SE method needs smaller mesh size compared to the SPME method, the runtime difference to achieve a certain level of accuracy is large. Note that the grid size is chosen to suit the FFT routines and to minimize runtime fluctuations.

\subsubsection*{Solve} 
In this context, \textit{solve} refers to the scaling of the Fourier transformed distributed charges on a uniform mesh. Again, the only difference in runtime of both methods in this step is due to the difference in the grid size. Referring to section \ref{section:runtime_estimate}, this step has a complexity of $\order{M^3}$ and, therefore, the scaling step is almost on the same order of computational complexity as the FFT/IFFT step.
\begin{figure}[htbp]
  \tikzset{mark size=1.5}
  \begin{subfigure}[b]{0.49\textwidth}
    \centering
    \includegraphics[width=\textwidth,height=.8\textwidth]{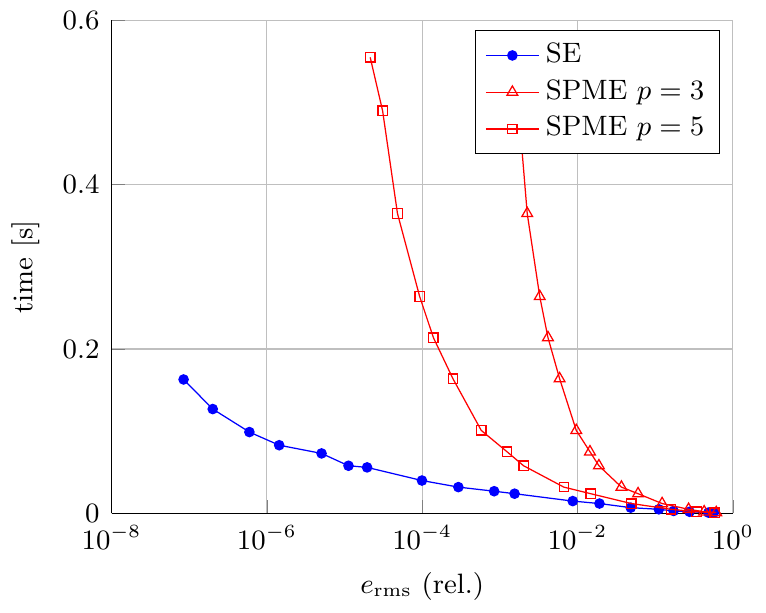}
  \end{subfigure}
\hfill
  \begin{subfigure}[b]{0.49\textwidth}
    \centering
    \includegraphics[width=\textwidth,height=.8\textwidth]{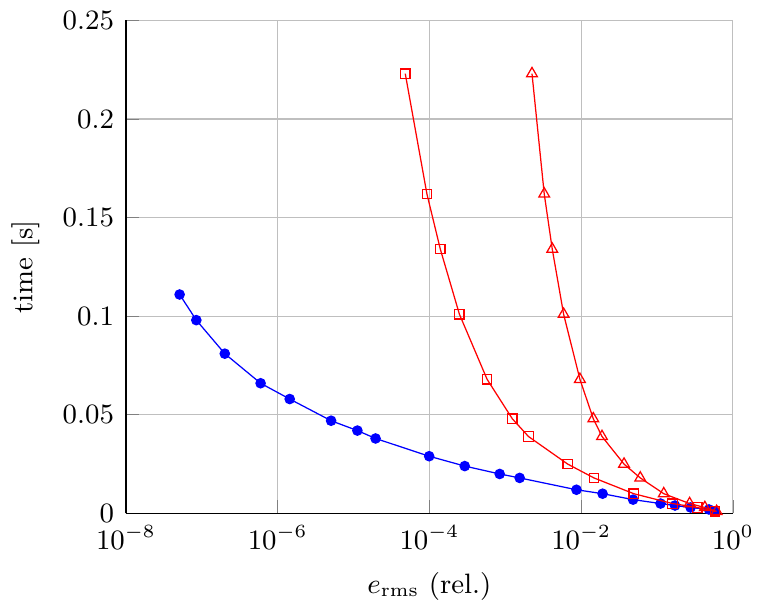}
  \end{subfigure}
  \caption{Runtime vs.\ relative rms error of (left) the FFT and (right) \textit{solve} routines. A uniform system with $N=81000$, $L=9.3222$ and $\xi=6.5$ is used. In the SE method, $P=16$ and $M=40,\ldots,192$ and in the SPME method $p=3$ and $p=5$ (from top to bottom) and $M=40,\ldots,350$. The red curves show the runtime of the SPME method and the blue curves represent the runtime of the SE method.}
  \label{fig:serial:se_spme3}
\end{figure}

\subsubsection*{Spreading and Gathering} 
Theoretically in the SPME method, it is possible to achieve any desired accuracy with all B-spline orders (even though it might not be practically feasible) but due to the low accuracy demand in MD, B-splines of order $3$ are frequently used. In the SE method, on the other hand, increasing the grid size while keeping $P$ fixed, reduces the error up to a certain level of accuracy. When the approximation error dominates truncation errors, to decrease the total error, $P$ has to be increased. 

The runtime for the SE method is dominated by the spreading and gathering steps. Hence,
when the total computational time is considered, the comparison of
the two methods is much more complex, as will be discussed next.

\subsection{SE-SPME runtime comparison}
\label{section:se_spme_time_comparison}
Our aim here is to examine sensitivity and behavior of both methods for different Ewald parameters. We use a system of $N=81\,000$ uniformly distributed particles in a box of length $L=9.3222$ and plot Fourier space runtime as a function of relative rms error, see figure \ref{fig:time_tol_se_spme_3_65}. For the SE method, $P$ is obtained using the approximation error bounds calculated in section \ref{section:error_se:approximation_error_force}. For the SPME method we use $p=3$ and 5. We observe that the runtime of the SE method is less sensitive to the Ewald parameter than the SPME method and therefore the optimal Ewald parameter can be found with less effort. Also, we note that for the case of $\xi=6.5$, error tolerance of $10^{-5}$ cannot be achieved with $p=3$ in a reasonable wall-clock time. 
\begin{figure}[htbp]
  \tikzset{mark size=1.5}
  \begin{subfigure}[b]{0.49\textwidth}
    \centering
    \includegraphics[width=\textwidth,height=.8\textwidth]{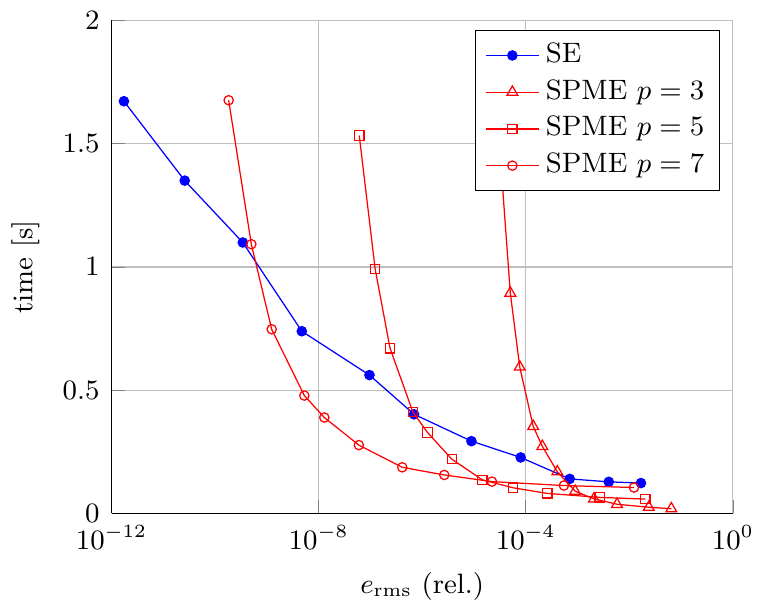}
  \end{subfigure}
  \hfill
  \begin{subfigure}[b]{0.49\textwidth}
    \centering
    \includegraphics[width=\textwidth,height=.8\textwidth]{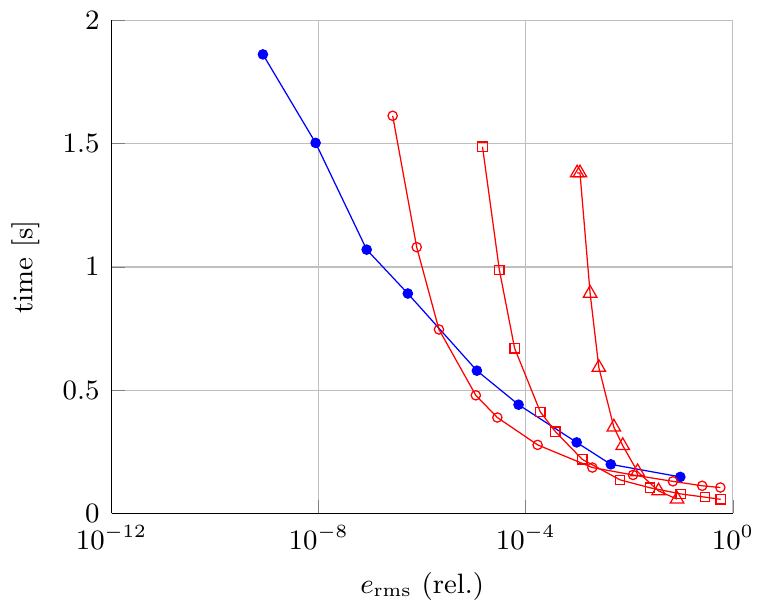}
  \end{subfigure}
  \caption{Fourier space runtime comparison of the SE and SPME methods using a single CPU core. (left) $\xi=3$, $P=8,\ldots,18$ and $M=40,\ldots,112$ in the SE method and $M=40,\ldots,400$ in the SPME method.(right) $\xi=6.5$, $P=8,\ldots,18$ and $M=80,\ldots,240$ in the SE method and $M=40,\ldots,400$ in the SPME method. A uniform system with $N=81\,000$ particles and $L=9.3222$ is used.}
  \label{fig:time_tol_se_spme_3_65}
\end{figure}



In the next experiment, in figure \ref{fig:se_spme_time_comparison:time_xi_se_spme_1e5} we compare Fourier space runtime as a function of $\xi$ for a relative rms error of $\approx10^{-5}$. For the SPME method, B-splines of order 3, 5 and 7 were used. For the SE method, as before, we obtain $P$ from the approximation error bound. The figure confirms that for the SPME method to be efficient, $\xi$ has to be kept small enough or $p$ has to be increased otherwise. For example, increasing $\xi$ from 3 to 6.5, the runtime of the SPME method with $p=7$ increases with a factor of 5 while this is less than a factor of 2 for the SE method. 
The cost of spreading/gathering and pre-computation steps depend only on $P$ (since $N$ is fixed). Also $P$ is uniquely determined by the error tolerance and, therefore, is constant for all $\xi$ in this test. Increasing $\xi$, however, increases the grid size and consequently the cost of the FFT and \textit{solve} steps. 
\begin{figure}[htbp]
  \tikzset{mark size=1.5}
  \centering
  \includegraphics[width=0.6\textwidth,height=.4\textwidth]{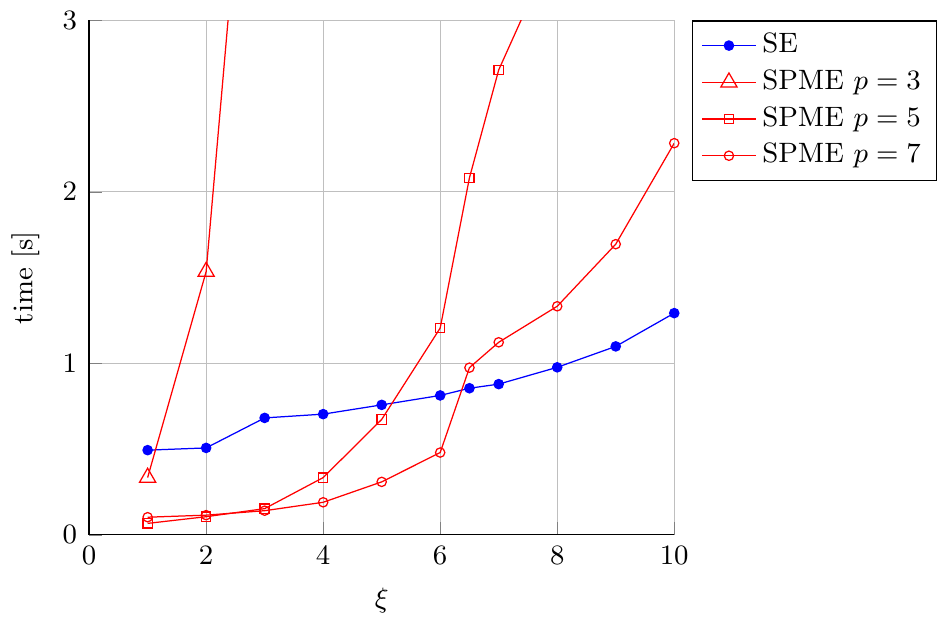}
  \caption{Fourier space runtime comparison of the SE and SPME methods using a single CPU core to achieve a relative rms error of $\approx10^{-5}$. The red lines are the runtime of the SPME method with  B-splines of orders $p=3,5,7$. A uniform system with $N=81000$ particles in a box of length  $L=9.3222$ is used.}
  \label{fig:se_spme_time_comparison:time_xi_se_spme_1e5}
\end{figure}

\section{Comparison of the SE and SPME methods - parallel implementation}
\label{section:parallel}
On parallel computers or while using heterogeneous architectures, a larger grid size in the 3D FFT algorithm will increase the cost of communications between processes. This cost can degrade the performance of the FFT-based methods and specially those with larger FFT grid requirement. Unlike FFTs, spreading/gathering and solve steps benefit more from parallel processing.  We use the GROMACS parallel paradigm \cite{gromacs_parallel} and implement a parallel version of the SE method. As has already been explained, two 3D FFTs are needed in the implementation of the SE and SPME methods. The parallel implementation of the 3D FFT in GROMACS requires also two redistributions (or four, considering also the 3D IFFT) of 3D grids that is clearly expensive. In figure \ref{fig:parallel:parallel_4_cores} we use a uniform system of $N=81\, 000$ particles with $L=9.3222$ and plot Fourier space runtime of the SE method with $P=8,\ldots,18$ and SPME method with $p=3,5,7$ for $\xi=3$ (left) and $\xi=6.5$ (right) using 4 CPU cores. Comparing figure \ref{fig:time_tol_se_spme_3_65} and \ref{fig:parallel:parallel_4_cores} discloses that for $\xi=6.5$, the intersection of the runtime curve for the SE and SPME ($p=7$) methods changes from an error level around $10^{-6}$ in the serial implementation to $10^{-5}$ in the parallel implementation with 4 CPU cores.

\begin{figure}[htbp]
  \tikzset{mark size=1.5}
  \begin{subfigure}[b]{0.49\textwidth}
    \centering
    \includegraphics[width=\textwidth,height=.8\textwidth]{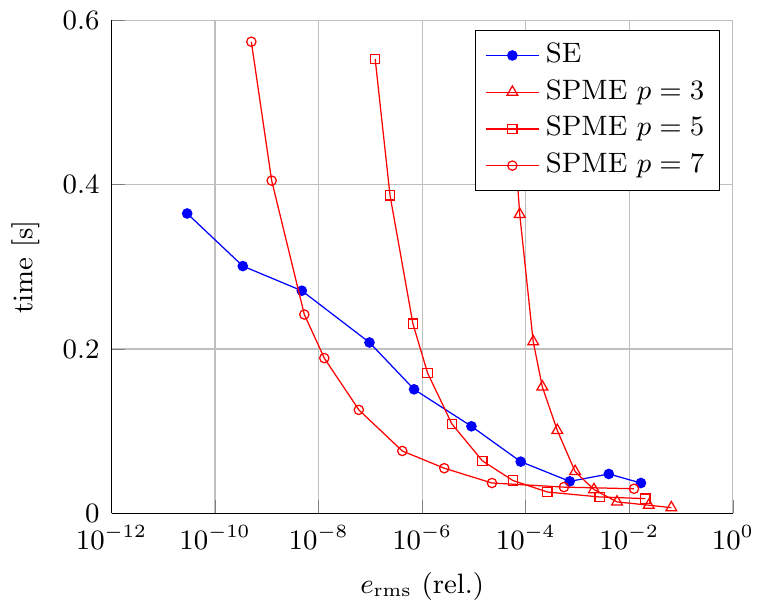}
\end{subfigure}
\hfill
  \begin{subfigure}[b]{0.49\textwidth}
    \centering
    \includegraphics[width=\textwidth,height=.8\textwidth]{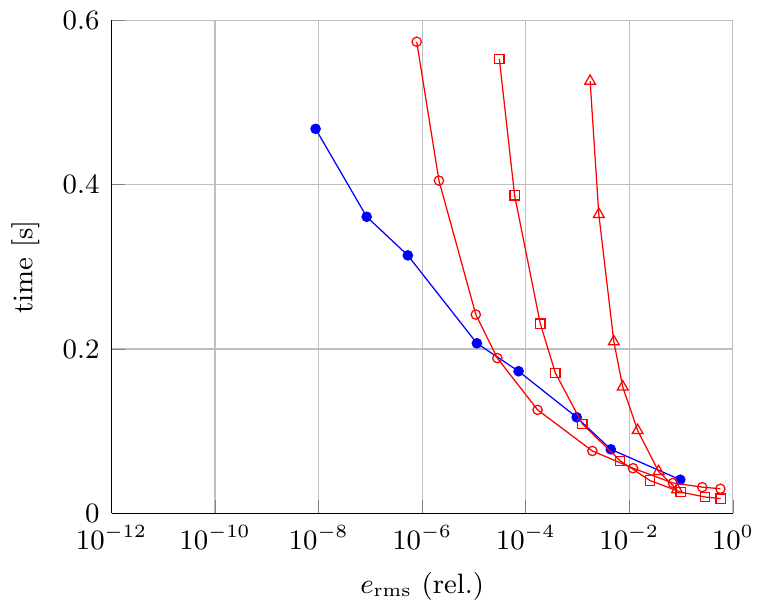}
\end{subfigure}

\caption{Fourier space runtime comparison of the SPME method of order $p=3,5,7$ and the SE method using 4 CPU cores. (left) $\xi=3$, $P=8,\ldots,18$ and $M=40,\ldots,112$ in the SE method and $M=40,\ldots,400$ in the SPME method. (right) $\xi=6.5$, $P=8,\ldots,18$ and $M=80,\ldots,240$ in the SE method and $M=40,\ldots,400$ in the SPME method. A uniform system with $N=81\,000$ particles and $L=9.3222$ is used.}
\label{fig:parallel:parallel_4_cores}
\end{figure}



Next, we generate a uniform system of $N=1\,029\,000$ particles within a box of size $L=22.2698$. This system is generated by replicating a uniform system of $N=3000$ particles in a box of length $L=9.3222$. As in the previous section and similar to the case with $N=81\,000$ uniformly distributed system of particles, we choose $\xi=6.5$. If we also consider the real space sum cost, runtime estimates in section \ref{section:runtime_estimate} can be used to obtain an almost optimal Ewald parameter, $\xi=4.184$. The result using 8 CPU threads is shown in figure \ref{fig:parallel:uniform_8_cores}. We observe that the SE method is still competitive for low error tolerances. In fact, the cross overs are decreased and occur at $\approx10^{-5}$ for $\xi=6.5$ and at $\approx10^{-6}$ for $\xi=4.184$ instead. In the right plot of figure \ref{fig:parallel:uniform_8_cores}, due to insufficient memory, the simulation can not be performed for $M>540$ (red curves) on our desktop machine.

\begin{figure}[htbp]
  \tikzset{mark size=1.5}
  \begin{subfigure}[b]{0.49\textwidth}
    \centering
    \includegraphics[width=\textwidth,height=.8\textwidth]{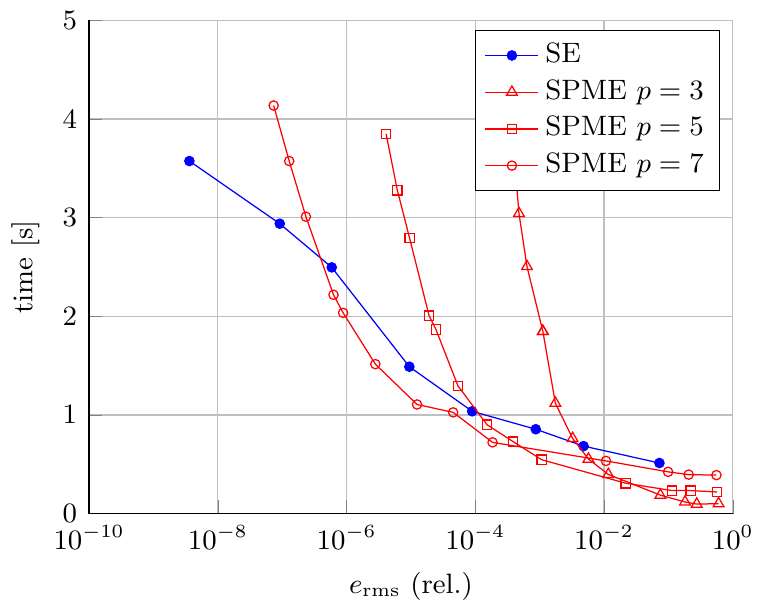}
\end{subfigure}
\hfill
  \begin{subfigure}[b]{0.49\textwidth}
    \centering
    \includegraphics[width=\textwidth,height=.8\textwidth]{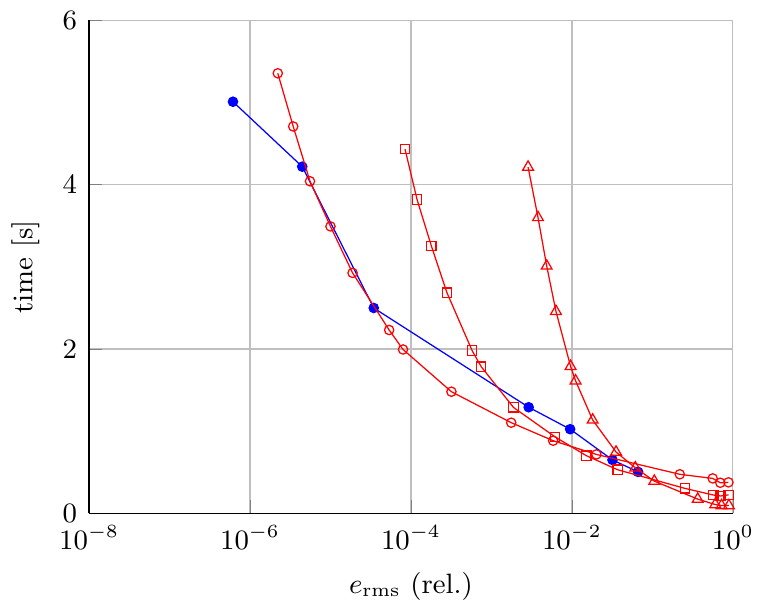}
\end{subfigure}
\caption{Fourier space runtime comparison of the SE and SPME methods using 8 CPU threads. A uniform system of $N=1\,029\,000$ particles in a box of length $L=22.2698$ is used. For the SE method $P=6,...18$ and $M=220,\ldots,360$, and for the SPME method,  $p=3,5,7$ and $M=50,\ldots,540$. (left) $\xi=4.184$, (right) $\xi=6.5$.}
\label{fig:parallel:uniform_8_cores}
\end{figure}



Figure \ref{fig:parallel:wall8_cores} shows a similar experiment using a cloud-wall system of $N=1\,012\,500$ particles in a box of length $L=150$. As in \cite{axel1}, we set $\xi=0.7002$ for the left and $\xi=0.8267$ for the right plot of figure \ref{fig:parallel:wall8_cores}. We should note that based on the Kolafa $\&$ Perram truncation error estimates described in section \ref{section:truncation}, Fourier space truncation error is controlled by $\xi L$ and not only $\xi$. Hence, the Ewald parameters are one order of magnitude smaller than the parameters in the previous simulations. Based on these plots we can conclude that, for systems with strong long-range interactions, the SE method can be competitive even for error tolerances below $10^{-5}$.
\begin{figure}[htbp]
  \tikzset{mark size=1.5}
  \begin{subfigure}[b]{0.49\textwidth}
    \centering
    \includegraphics[width=\textwidth,height=.8\textwidth]{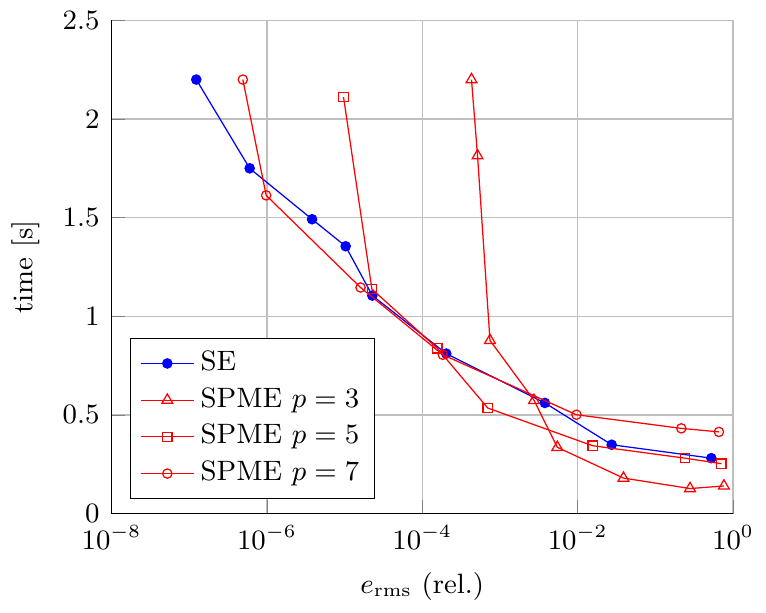}
\end{subfigure}
\hfill
  \begin{subfigure}[b]{0.49\textwidth}
    \centering
    \includegraphics[width=\textwidth,height=.8\textwidth]{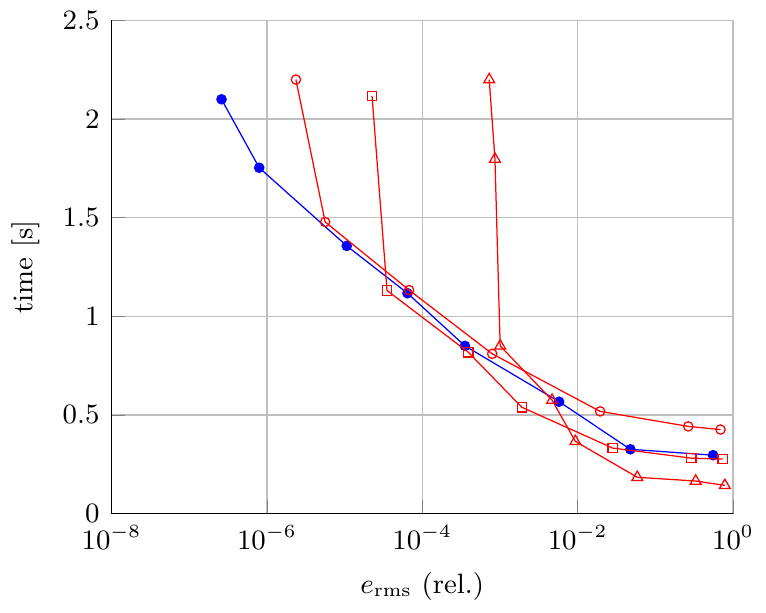}
\end{subfigure}
\caption{Fourier space runtime comparison of the SE and SPME methods using 8 CPU threads. A cloud-wall system with $N=1\,012\,500$ particles in a box of length $L=150$ is used. For the SE method $P=6,\ldots,16$ and for the SPME method  $p=3,5,7$. Also $M=32,\ldots,448$ for both methods. (left) $\xi=0.7002$, (right) $\xi=0.8267$.}
\label{fig:parallel:wall8_cores}
\end{figure}


\section{Further optimization}
\label{section:opt}
GROMACS can be accelerated with SIMD extensions, e.g. AVX $\&$ AVX2. In the SE method implementation, the spreading/gathering routines for $P=8,16$ and 24 can be accelerated using AVX and AVX2 instructions if the necessary conditions are met. Our desktop processor supports AVX instructions, but for a benchmark with AVX2, we will perform an experiment on Beskow supercomputer.

We start by a cloud-wall system of $N=1\,012\,500$ particles with box size $L=150$ and set $\xi=0.8267$. First, we choose $P=8,16$ and plot Fourier space runtime for the SE method as a function of total grid size $M^3$. Figure \ref{fig:avx:avx_sse_comparison} (left) shows that using AVX instructions, the total runtime is smaller and this reduction is more pronounced when $P=16$.
\begin{figure}[htbp]
  \tikzset{mark size=1.5}
  \begin{subfigure}[b]{0.49\textwidth}
    \centering
    \includegraphics[width=\textwidth,height=.8\textwidth]{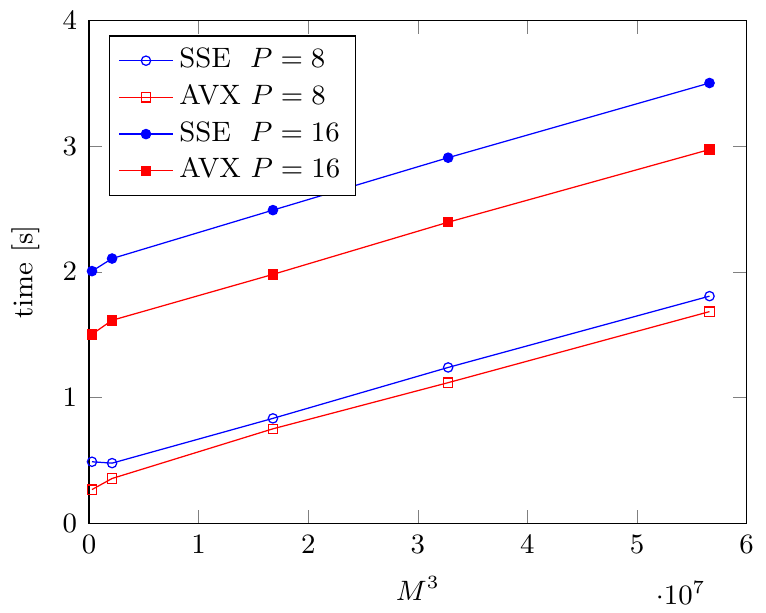}
\end{subfigure}
\hfill
  \begin{subfigure}[b]{0.49\textwidth}
    \centering
    \includegraphics[width=\textwidth,height=.8\textwidth]{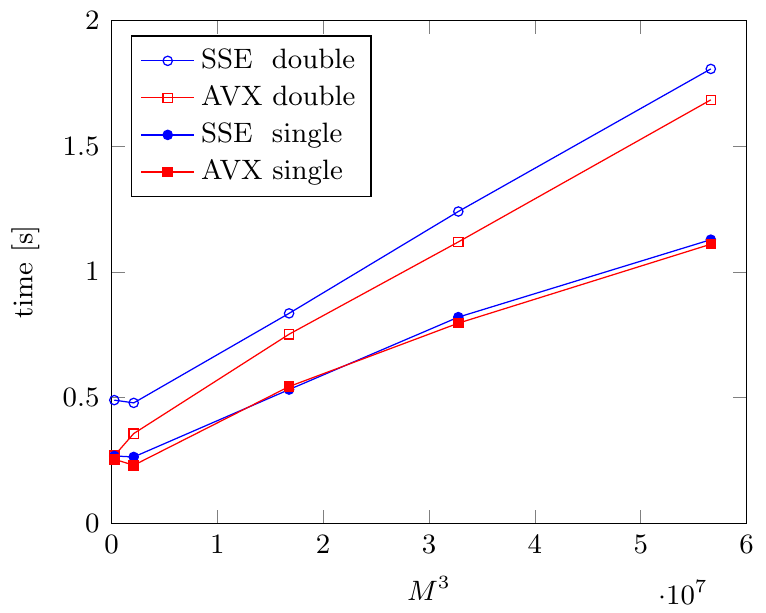}
\end{subfigure}
\caption{Fourier space runtime of the SE method as a function of the total grid size $M^3$. A cloud-wall system of $N=1\,012\,500$ particles in a box of length $L=150$ with $\xi=0.8267$ is used. Also $M=64,\ldots,400$. (left) SSE-AVX comparison with $P=8$ and 16 in double precision. (right) Single and double precision comparison using 8 CPU threads with $P=8$.}
\label{fig:avx:avx_sse_comparison}
\end{figure}

So far we have performed all experiments in double precision, but as a realistic comparison for low accuracy demands, single precision has to be used. We consider the same system as in the previous example with the same Ewald parameter and set $P=8$. The routine for $P=8$ can be accelerated using AVX and SSE instructions. Figure \ref{fig:avx:avx_sse_comparison} (right) shows a runtime comparison of the SE method in single and double precision as a function of the total grid size $M^3$. We clearly gain a factor of 2 using single precision, however, the difference between using AVX and SSE is minor in this case.

Now, we can compare the runtime of both SE and SPME methods using AVX and in single precision. We choose the same system and Ewald parameter as before and use $p=3$ and 5 for the SPME method and $P=8$ for the SE method. Figure \ref{fig:avx:avx_single_pme_se} shows the relative rms error as a function of runtime. Note that, in contrast to the right plot in figure \ref{fig:parallel:wall8_cores}, the result here is obtained with a fixed $P$. Therefore, the blue curve (SE method) levels out for large $M$ and to achieve a higher accuracy, $P$ has to be increased.
\begin{figure}[htbp]
  \tikzset{mark size=1.5}
  \centering 
  \includegraphics[width=0.5\textwidth,height=.4\textwidth]{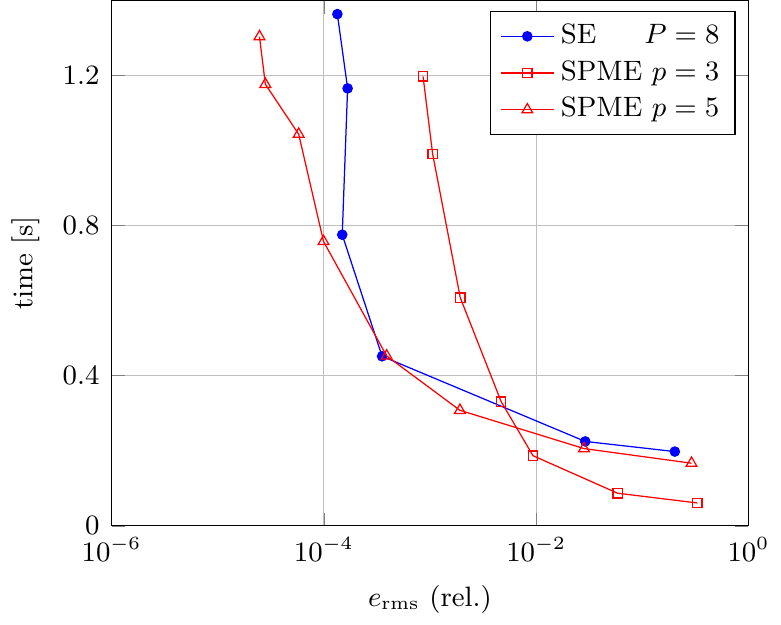}
  \caption{Fourier space runtime comparison of the SE ($P=8$) and SPME ($p=3,5$) methods with AVX intrinsics enabled in single precision using 8 CPU threads. A cloud-wall system of $N=1\,012\,500$ particles in a box of length $L=150$ with $\xi=0.8267$ is used. Also $M=64,\ldots,400$.}
  \label{fig:avx:avx_single_pme_se}
\end{figure}

To complete our benchmarks and investigate the performance, we compare both methods on Beskow supercomputer at KTH Royal Institute of Technology, Sweden. Beskow is a Cray XC40 system, based on Intel Haswell processors and consists of 53632 compute cores. Haswell processors support AVX2 as well as \textit{Fused Multiply Add} (FMA) instructions. This is more favorable for the SE method since there are more multiply-add operations in the spreading/gathering step compared to the SPME method.

We use a uniform system of $N=1\,029\,000$ particles with $L=22.2698$ and set $\xi=6.5$. We, furthermore, set an error tolerance of $10^{-5}$. To achieve a relative rms error less than the set tolerance, we choose $P=16$ and $M=320$ for the SE method and $p=7$ and $M=500$ for the SPME method. Note that with $P=10$, the same error tolerance can be reached, but the performance is degraded. 

GROMACS is also compiled and accelerated with AVX2 SIMD instructions. Figure \ref{fig:beskow} shows the runtime comparison of both methods for up to $256$ MPI ranks. Even though the SE method works only marginally better than the SPME method this difference is more significant if the time iteration is to be considered. In the right plot of figure \ref{fig:beskow}, we break down the cost of FFT and spreading/gathering steps for both methods. It is worth mentioning that, for a fixed grid size, using a very large number of cores can reduce the performance since the cost of communications in the FFTs might dominate the computational time. This is more evident in the SE method (blue circle-dotted line) as the grid is smaller.

\begin{figure}[htbp]
  \tikzset{mark size=1.5}
  \begin{subfigure}[b]{0.49\textwidth}
    \centering
    \includegraphics[width=\textwidth,height=.8\textwidth]{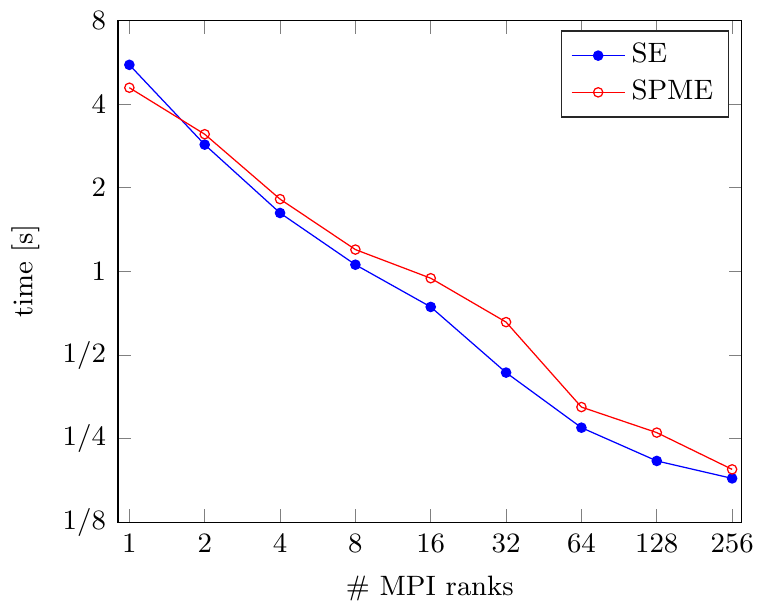}
\end{subfigure}
\hfill
  \begin{subfigure}[b]{0.49\textwidth}
    \centering
    \includegraphics[width=\textwidth,height=.8\textwidth]{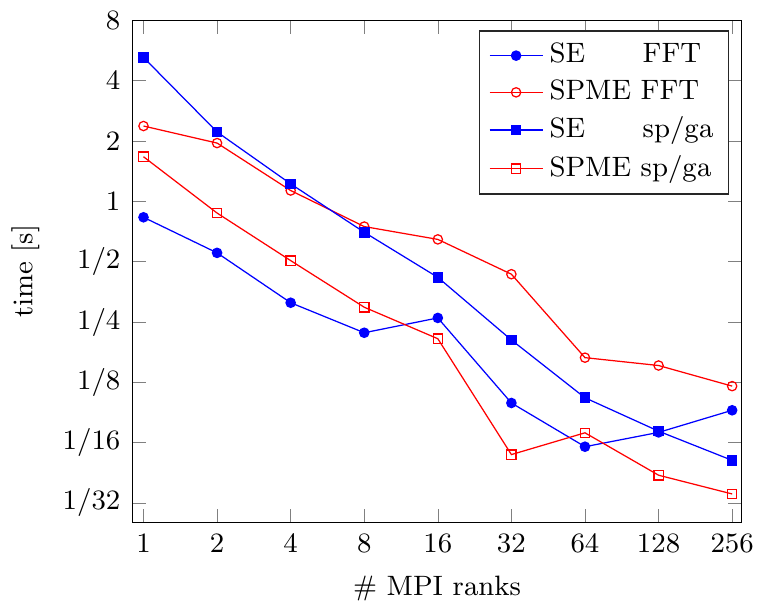}
\end{subfigure}
\caption{(left) Runtime and (right) breakdown runtime of the SE and SPME methods in computing the electrostatic force to achieve a relative rms error of $\approx10^{-5}$. Computation is done on Beskow supercomputer in single precision and with up to 256 MPI ranks. Also $P=16$, $M=320$ for the SE method and $p=7$, $M=500$ for the SPME method and $\xi=6.5$. A uniform system of $N=1\,029\,000$ particles with $L=22.2698$ is used.}
\label{fig:beskow}
\end{figure}

\section{Conclusions}
In this work we compared our implementation of the SE method
introduced by Lindbo $\&$ Tornberg in \cite{se} and a highly optimized
version of the SPME method \cite{spme} in GROMACS \cite{gromacs} version 5.1.
The two methods belong to the family of Particle Mesh Ewald methods,
and are similar in their structure. However, by the different choice
of window functions - suitably scaled and truncated Gaussians vs
$p$th order B-splines ($p=3,5,7$), significant differences do arise.

Approximation errors in the SE method are essentially
independent of the FFT grid size and can be controlled by increasing
the number of points of support in the Gaussian while simultaneously
rescaling it. Whereas in the SPME method approximation errors decay algebraically as $h^p$ as
the grid size h is decreased. This means that smaller FFT grids can
be used in the SE method as compared to the SPME method. On the other
hand, the gridding and spreading steps that use the window functions
for interpolation, are more costly for the SE method.

Both methods can use the same routine for the evaluation of the real
space sum, so the comparison really entails the evaluation of the
$k$-space sum. However, it is desirable to balance the work done for the
two parts, and hence the implementation for the real space sum will
effect the choice of the Ewald parameter $\xi$. As we have seen in the
results section (figure \ref{fig:se_spme_time_comparison:time_xi_se_spme_1e5}), 
the $k$-space runtime for the SPME method is much more sensitive to the 
choice of $\xi$ than the SE method.

The comparison between the two methods will depend on the specific
problem, the value of $\xi$, the accuracy requirement, and also on the
implementation. Generally speaking, the SE method will be most
efficient if high accuracy is required, whereas the SPME method will
be most efficient for low accuracy demands. The choice of $p=3$ for the
B-splines is essentially only competitive with the other choices if
a very low accuracy is sufficient, and it also shows the highest
sensitivity in the choice of $\xi$, which makes it less attractive to
work with.

Our benchmarks include system of particles with periodic boundary
conditions and different number of particles, box sizes and
topologies. We ran GROMACS on a regular desktop
machine with an Intel Core i7-3770 CPU as well as up to 256 MPI ranks
on a supercomputers with AVX2 support.

Both algorithms were accelerated using SIMD instructions in single
and double precision. The spreading/interpolating steps can be done by
either computing the effect of all particles on a specific grid point
or by computing all the grid values for one particle at the time. In
our implementation, we choose the latter case. From the computer
science point of view, this case has the benefit of writing on a
consecutive memory locations. It is also advantageous for the
distributed memory architectures though it should be revised again
while using Graphics processing Units (GPU).

Our benchmarks show that the typical relative rms error below
which the SE method becomes more efficient than the SPME method is
around $10^{-6}$. We however observed that for some cases where the system
is large enough and long range interactions are strong, it can be
competitive for accuracies around $10^{-5}$, and whenever SIMD instructions
are used in single precision, both methods are comparable
even for low accuracy demands. On the Beskow computer, a runtime
comparison for a system of around $10^{6}$ particles at a relative error
level of $10^{-5}$ showed the SE method to be somewhat faster than the
SPME method already from 2 MPI ranks, scaling up to 256. 

Already at this level of optimization, we can see that the SE method
can be competitive with the SPME method, and we expect to be able to
make further enhancements and accelerations of the method. As of yet,
the implementation of the SE method is not available as an integrated
part of GROMACS, but is publicly available at \cite{gromacs_se_github}.

\section*{Acknowledgement}
This work has been supported by the the Swedish Research Council under grant no. 2011-3178 and by the Swedish e-Science Research Center. The authors gratefully acknowledge this support.

\bibliographystyle{abbrvnat_mod}
\bibliography{library}

\end{document}